\documentclass[12pt]{amsart}
\usepackage{amsmath}
\usepackage{csquotes}
\usepackage{calrsfs}
\usepackage{worldflags}
\DeclareMathAlphabet{\pazocal}{OMS}{zplm}{m}{n}
\newcommand{\Xx}{\pazocal{X}}
\newcommand{\Yy}{\pazocal{Y}}
\usepackage{enumerate}
\usepackage{csquotes}
\usepackage{amsxtra}
\newcommand{\quotes}[1]{``#1''}
\usepackage{scalerel}
\usepackage{hyperref}
\usepackage{enumitem}
\usepackage{floatrow}
\usepackage[new]{old-arrows}
\usepackage{amssymb}
\usepackage{tikz-cd}
\theoremstyle{plain}
\newtheorem{Theo}[subsubsection]{Theorem}
\newtheorem{fact}[subsubsection]{Fact} 

\newtheorem{Theor}{Theorem}
\newtheorem{Prop}[subsubsection]{Proposition}
\newtheorem{problem}[subsubsection]{Problem}

\newtheorem*{Theorem}{Theorem}

\newtheorem{cor}[subsubsection]{Corollary}

\newtheorem{Conj}{Conjecture}
\newtheorem{question}[subsubsection]{Question}

\theoremstyle{definition}

\newtheorem{Lemm}[subsubsection]{Lemma} 

\newtheorem{definition}[subsubsection]{Definition}
\theoremstyle{remark}
\newtheorem{notation}[subsubsection]{Notation}
\newtheorem{Obs}[subsubsection]{Observation}
\newtheorem{rem}[subsubsection]{Remark}
\newtheorem{remark}[subsubsection]{Remark}
\newtheorem{example}[subsubsection]{Example}

\newtheorem{ex}[subsection]{Exercise}

\newcommand{\nc}{\newcommand}

\nc{\bp}{\begin{Prop}}
\nc{\ep}{\end{Prop}}
\nc{\ssn}{\pagebreak \section}
\nc{\df}{{\bf Definition}\ }
\nc{\bl}{\begin{Lemm}}
\nc{\el}{\end{Lemm}}
\nc{\Pic}{\text{Pic}}
\nc{\bex}{\begin{ex} \rm}
\nc{\eex}{\end{ex}}
\nc{\bt}{\begin{Theo}}
\nc{\et}{\end{Theo}}
\nc{\bq}{\begin{question}}
\nc{\eq}{\end{question}}
\nc{\bc}{\begin{cor}}
\nc{\ec}{\end{cor}}
\nc{\bob}{\begin{Obs}}
\nc{\eob}{\end{Obs}}
\nc{\N}{\mathbb{N}}
\nc{\Q}{\mathbb{Q}}
\nc{\Z}{\mathbb{Z}}
\nc{\Ss}{{\mathbb{S}}}
\nc{\Cc}{{\mathbb{C}}}
\nc{\F}{{\mathbb{F}}}
\nc{\Oo}{\mathcal{O}}
\nc{\Qq}{\mathbb{Q}}
\nc{\ulim}{\text{ulim}\ }
\nc{\Hom}{\text{Hom}}
\nc{\Ext}{\text{Ext}}
\nc{\Tor}{\text{Tor}}
\nc{\Ob}{\text{Ob}}
\nc{\id}{\text{id}}
\nc{\ad}{\text{ad}}
\nc{\an}{\text{an}}
\nc{\rig}{\text{rig}}
\nc{\Spa}{\text{Spa}}
\nc{\ZR}{\text{ZR}}
\nc{\Hilb}{\text{Hilb}}
\nc{\supp}{\text{supp}}
\nc{\Spf}{\text{Spf}}
\renewcommand{\mkbegdispquote}[2]{\itshape}
 \nc{\Spec}{{\mathop{\operatorname{\rm Spec}}}}
  \nc{\spec}{{\mathop{\operatorname{\rm Spec}}}}

\usepackage{yhmath}

\usepackage[utf8]{inputenc}

\title{On geometrically $C_1$ fields}
\author{Konstantinos Kartas}
\thanks{During this research, the author received funding from the European Union's Horizon 2020 research and innovation programme under the Marie Sk\l odowska-Curie grant agreement No 101034255\worldflag[width=1.5mm, length=1.5mm]{EU} and was also supported by the program GeoMod ANR-19-CE40-0022-01 (ANR-DFG)}
\begin{document}

\maketitle
\begin{abstract}
A field $k$ is called geometrically $C_1$ if every smooth projective separably rationally connected $k$-variety has a $k$-rational point. Given a henselian valued field of equal characteristic $0$ with divisible value group, we show that the property of being geometrically $C_1$ lifts from the residue field to the valued field. We also prove that algebraically maximal valued fields with divisible value group and finite residue field are geometrically $C_1$. In particular, any maximal totally ramified extension of a local field is geometrically $C_1$. 
%and the same is true for the Hahn series field $\F_q(\!(t^{\Q})\!)$.
%Let $(K,v)$ be a henselian valued field of equal characteristic $0$ with divisible value group. Then, $K$ is geometrically $C_1$ if and only if $k$ is geometrically $C_1$.
%\\
%(2) Prove that $\Q_p(p^{1/p^{\infty}})$ is $C_2$.
\end{abstract}

\setcounter{tocdepth}{1}
\tableofcontents
%\pagebreak
\section{Introduction}
%Deformation invariance of SRC; Question 3.9. — For a smooth, proper scheme over a DVR, if the
%generic fiber is separably rationally connected, is the closed fiber separably
%rationally connected? https://aif.centre-mersenne.org/item/10.5802/aif.3495.pdf
%Probably SRC is not deformation invariant: (see https://arxiv.org/pdf/math/0205148.pdf) Kollar assumes Let $X$ is a
%smooth proper morphism and that $X_k$ is SRC (probably assuming that the generic fiber is SRC is not enough). But maybe that is because $X_k$ is not guaranteed to be geometrically irreducible.
A field $k$ is called $C_i$ ($i\in \N$) if every non-constant homogeneous polynomial $f(X_0,...,X_n)\in k[X_0,...,X_n]$ of degree $d$ with $ d^i\leq n$ has a non-trivial zero over $k$. In geometric terms, we require that
every hypersurface $X\subseteq \mathbb{P}_k^n$ of degree $d$ satisfying $d^i \leq n $ has a $k$-rational point. 
Throughout the paper, we focus on $C_1$ fields, sometimes also referred to as quasi-algebraically closed.
%It has been observed that $C_1$ fields tend to admit rational points in more varieties than one might initially expect from their definition. 
It has been observed that $C_1$ fields tend to have rational points in many more varieties than just hypersurfaces $X\subseteq \mathbb{P}_k^n$ of degree $d\leq n$, especially varieties that are in some sense close to being rational. These include geometrically rational varieties, geometrically unirational varieties and more generally varieties which contain lots of rational curves, the so-called \textit{rationally connected varieties} (see IV, \S 3 \cite{kollarbook}):
\begin{definition}
%[Koll\'ar-Miyaoka-Mori]
%[Definition 3.2.3 \cite{kollarbook}]
%Koll\'ar-Miyaoka-Mori]
%\begin{enumerate}[label=(\roman*)]
%\item Let $k$ be an algebraically closed field. 
A $k$-variety $X$ is called rationally connected (resp. separably rationally connected) if there is a $k$-variety $B$ and a morphism $F:B\times \mathbb{P}^1\to X$ such that  the induced morphism
$$B\times \mathbb{P}^1\times \mathbb{P}^1 \to X\times X: (b,t,t')\mapsto (F(b,t),F(b,t'))$$
is dominant (resp. dominant and separable). 
%\item Let $k$ be an arbitrary field. A $k$-variety $X$ is called rationally connected if $X_{\overline{k}}$ is rationally connected. 
%\end{enumerate}
\end{definition}
In other words, there is an algebraic family of proper rational curves such that for almost any $(x,x')\in X\times X$, there is a curve in the family joining $x$ and $x'$. 
%For instance, it was shown that $C_1$ fields admit rational points on Severi-Brauer varieties 
%since $C_1$ fields have trivial Brauer group. It also includes 
%and certain rational surfaces by an early result of Manin. 
%In a letter to Grothendieck \cite{serregro}, Serre had tentatively suggested that perhaps every $\Oo$-acyclic variety (viz. $H^i(X,\Oo_X)=0$ for $i>0$) over a $C_1$ field $k$ 
%of cohomological dimension $1$ 
%has a $k$-rational point---although he immediately adds that such a strong statement is rather unlikely. 
%Later, he amended his statement by replacing cohomological dimension $1$ with the stronger assumption that $k$ is $C_1$. 
%Indeed, this was proven to be too optimistic since it already fails for some of the most common $C_1$ fields, such as function fields of curves over algebraically closed fields (see \cite{GHMS}). On the positive side, a lot of evidence gradually was accumulated over the years, which suggests that $C_1$ fields have rational points on \textit{rationally connected varieties}. [post list here?]
%
%have , which share many common properties with geometrically rational varieties but are in many regards better behaved:
%It may be convenient to look not just at hypersurfaces but at a more general class of varieties. 
\begin{definition}[Koll\'ar]
A field $k$ is called \textit{geometrically }$C_1$ if every smooth projective separably rationally connected $k$-variety has a $k$-rational point. 
\end{definition}
In \S \ref{c1fieldsandvariants}, we study systematically geometrically $C_1$ fields as well as other variants of $C_1$ fields. Each of these variants forms an elementary class which is $\forall \exists$-axiomatizable in the language of rings, similar to $C_1$ fields. This fact is used extensively throughout the paper.
%Other variants of the $C_1$ property are studied in \S \ref{c1fieldsandvariants}.
%https://webusers.imj-prg.fr/~leila.schneps/grothendieckcircle/Letters/GS.pdf pg. 172 
%In SGA5 pg. 134 Serre corrects it by replacing cd 1 with C_1 https://wstein.org/sga/circle/SGA5.pdf
Geometrically $C_1$ fields of characteristic $0$ are $C_1$ by a theorem of Hogadi-Xu \cite{hogadi}. Conversely, the Lang-Manin conjecture (or $C_1$-conjecture) predicts that every $C_1$ field is geometrically $C_1$. It has been verified for several $C_1$ fields, see \cite{esnault2,grabberstarr,dejongstarr}. We refer the reader to a recent survey article of Esnault \cite{esnaultc1} on this problem. 
%It is not known if discrete henselian valued fields of mixed characteristic with algebraically closed residue field are geometrically $C_1$.  The only known proof in equal characteristic passes through the function field case and does not seem to shed any light on the mixed characteristic case. 
%In general, it is natural to look for algebraic extensions of $\Q_p$ which are geometrically $C_1$. Such extensions must necessarily be infinite since finite extensions of $\Q_p$ even fail to be $C_i$, for any $i\in \N$ (see \cite{Terj2,arkhipov,alemu}). In this paper, we obtain some of the first few natural examples. 

We propose the following transfer principle which says that, for certain valued fields, the property of being \quotes{geometrically $C_1$} lifts from the residue field to the valued field (and vice versa):
\begin{Conj}
Let $(K,v)$ be an algebraically maximal valued field with divisible value group and perfect residue field $k$. Then: 
$$K\mbox{ is geometrically }C_1\iff k \mbox{ is geometrically }C_1 $$
\end{Conj}
In this paper, which is the first in a series, we provide some evidence for such a statement. It is also an interesting question whether a similar transfer principle holds for $C_1$ (or even $C_i$). 
%The original motivation for this work w
We note that, by the main result of \cite{JK}, such statements imply analogous transfer principles between a perfectoid field and its tilt. This is the original motivation for our work and will be the subject of the sequel.
%Assuming the Lang-Manin conjecture, this should indeed be the case.
\subsection{Main results}
The situation is rather clear in equal characteristic $0$:
%The transfer of \quotes{geometrically $C_1$} holds 
\bt \label{equal0}
 Let $(K,v)$ be a henselian valued field of equal characteristic $0$ with divisible value group and residue field $k$. Then:
 $$K\mbox{ is geometrically }C_1\iff k \mbox{ is geometrically }C_1 $$
\et 
%The key ingredient is the tHogadi-Xu \cite{hogadi}
%Suppose $K=\bigcup_{n\in \N} k(\!(t^{1/n})\!)$ is the Puiseux series field over $k$. The general case can be reduced to the above by Ax-Kochen/Ershov in equal characteristic $0$. The key ingredient is a result of Hogadi-Xu \cite{hogadi}, 
%the existence of semistable models and the theorem of  without any reference to valuation theory. For such a reduction, it is crucial that geometrically $C_1$ fields form an elementary class (see \S \ref{c1fieldsandvariants}), a fact that is systematically used  throughout the paper and which is largely implicit in \cite{duesler}. 
For example, if $k$ is a geometrically $C_1$ field of characteristic $0$, then the Puiseux series field $K=\bigcup_{n\in \N} k(\!(t^{1/n})\!)$ is geometrically $C_1$. 
%$k$ is a , then 
%using the theorem of Graber-Harris-Starr \cite{grabberstarr}, which says that function fields of curves over algebraically closed fields of characteristic $0$ are geometrically $C_1$, one gets the following: 
%\bc 
%Let $k$ be an algebraically closed field of characteristic $0$. The Puiseux series field $\bigcup_{n\in \N} k(z) (\!(t^{1/n})\!)$ is geometrically $C_1$ and also $C_1$.
%\ec
We stress however that Theorem \ref{equal0} is also valid for valued fields of arbitrary (possibly infinite) rank.

%This relies on , which says that function fields of curves over algebraically closed fields of characteristic $0$ are geometrically $C_1$.
In mixed and positive characteristics, we were only able to verify such a transfer principle for specific geometrically $C_1$ residue fields $k$. In this paper, we deal with the case where $k$ is a finite field (or algebraic over a finite field):
\bt  \label{hendefdiv1}
Let $(K,v)$ be an algebraically maximal valued field with divisible value group. Suppose that $k$ is algebraic over a finite field. Then $K$ is geometrically $C_1$ and also $C_1$.
\et   
We present two applications. Even the fact that the fields below are $C_1$ appears to be a new observation.
%We prove that any maximal totally ramified extension of $\Q_p$ is geometrically $C_1$. More generally:
\bc  \label{maxtot}
Any maximal totally ramified extension of a non-archimedean local field is geometrically $C_1$ and also $C_1$.
\ec  
This is, in a way, complementary to Lang's theorem \cite{langc1} which says that the maximal unramified extension of any non-archimedean local field is $C_1$. We note that it is a prominent open problem whether $\Q_p^{ur}$ is geometrically $C_1$; see \cite{esnaultindex,duesler,pieropan} for partial results in this direction. 

I was later informed by O. Wittenberg that he knew how to prove a weaker form of Corollary \ref{maxtot} by using de Jong's alterations as in \cite{witt}.
%, using the equivariant version of de Jong's theorem on alterations (used just
%  as it is used in [Wit15] immediately after the proof of Lemme 5.3, 
%. 
Namely, he was able to show that every smooth projective rationally connected variety over a local field has a rational point in a finite totally ramified extension. In contrast, Corollary \ref{maxtot} implies the existence of a rational point in a finite totally ramified extension which is contained within an a priori chosen maximal totally ramified extension.
%This is a weaker version because it does not control the totally ramified extension

Our second application is the following:
\bc 
Let $\F_q$ be a finite field and $\Gamma$ be a divisible ordered abelian group. Then the Hahn series field $\F_q(\!(t^{\Gamma})\!)$ is geometrically $C_1$ and also $C_1$.
\ec
One cannot replace Hahn series with Puiseux series in the above statement---it is crucial to work with algebraically maximal rather than henselian fields when the residue characteristic is positive. This is somewhat reminiscent of the fact that the field of Puiseux series over an algebraically closed field of positive characteristic fails to be algebraically closed due to the presence of Artin-Schreier extensions.

%
%\subsection{Rational points over tame fields}
%Given a valued field $(K,v)$ of rank $1$ and a proper $K$-variety $X$, an $\Oo_K$-model of $X$ is a flat, proper $\Oo_K$-scheme  of finite presentation whose generic fiber is $X$. 
%If $X(K)\neq \emptyset$ and $\Xx$ is any $\Oo_K$-model of $X$, then $\Xx(k)\neq \emptyset$ by the valuative criterion of properness. We prove a converse if the valuation is sufficiently nice:
\subsection{Elements in the proofs}
The main steps in the proofs of Theorems \ref{equal0} and \ref{hendefdiv1} are as follows: 
\begin{enumerate}
\item We first prove that, for certain valued fields $(K,v)$, given a $K$-variety $X$, we have $X(K)\neq \emptyset$ if and only if $\Xx(k)\neq \emptyset$ for each $\Oo_K$-model $\Xx$ of $X$. This reduces our task of showing that $K$ is geometrically $C_1$ to the task of finding $k$-points on degenerations of rationally connected varieties. 

\item It is a general principle that geometrically $C_1$ fields also tend to have rational points in degenerations of rationally connected varieties. This is true both for fields of characteristic $0$ and for finite fields, which are the residue fields under consideration here. 

\item From the above, it follows that $K$ is geometrically $C_1$. In case $K$ has positive characteristic, some extra work is needed to show that $K$ is $C_1$.  
\end{enumerate}
We elaborate more on these steps below.
%degenerations of rationally connected varieties over the residue field.
\subsubsection{Rational points over tame fields with divisible value group}
%We will the following criterion for the existence of a $K$-point, which 
 
\bt \label{criterion}
Let $(K,v)$ be an algebraically maximal valued field with divisible value group of rank $1$ and perfect residue field $k$. Let $X$ be a proper $K$-variety. Suppose that $\Xx(k)\neq \emptyset$, for each $\Oo_K$-model $\Xx$ of $X$. Then $X(K)\neq \emptyset$.
%\begin{enumerate}[label=(\roman*)] 
%\item 
%\item 
%\end{enumerate}
\et  
%In equal characteristic $0$, there is a purely geometric proof using the existence of a strict semi-stable $\Oo_K$-model $\Xx$ for $X$. 
Note that this statement is clear in favorable situations, e.g., if some $\Oo_K$-model $\Xx$ is smooth. We would then get that $X(K)\neq \emptyset$ by Hensel's Lemma. Unfortunately, in mixed and positive characteristics, we cannot really control the singularities of $\Oo_K$-schemes---except in very small dimensions. Nevertheless, there are certain results in the model theory of valued fields which are known in arbitrary characteristic and can be utilized in this situation. Proposition \ref{criterion} is essentially a geometric version of an Ax-Kochen/Ershov principle for existential closedness for tame fields
%existential  Ax-Kochen/Ershov principles 
due to F.-V. Kuhlmann \cite{Kuhl}.
%There are  type results, due to , which suffice for the existence of a $K$-point on $X$. 
In order to deduce the above geometric version from Kuhlmann's result, we first need to shift from $\Oo_K$-models to valuations: The $\Oo_K$-models of $X$  form an inverse system and the inverse limit---more precisely, the inverse limit of their formal completions---is isomorphic to a space of valuations on $X$, namely the adic space $X^{\ad}$. By a compactness argument, this inverse limit admits a $k^*$-point, for any sufficiently saturated elementary extension $k^*$ of $k$. This yields
%, namely 
%$$\varprojlim \Xx (k) \neq \emptyset $$ 
%A $k$-rational point of $X^{\ad}$
% a $k$-rational point in the adic space $X^{\ad}$ associated to $X$, hence 
a scheme-theoretic point $\xi\in X$ and a valuation $v_{\xi}$ on $\kappa(\xi)$ whose residue field $k'$ embeds in $k^*$ over $k$. In particular, we get that $k$ is existentially closed in $k'$. Since the value group of $K$ is divisible, it is existentially closed in the one of $v_{\xi}$, even though the latter will typically have higher rank. By Kuhlmann's Ax-Kochen/Ershov principle for existential closedness, one gets that $K\preceq_{\exists} \kappa(\xi)$. In particular, $X(K)\neq \emptyset$.  We spell out this argument in more detail in \S \ref{adicspaces}.
\subsubsection{Degenerations of rationally connected varieties}
Over a field of characteristic $0$, Hogadi-Xu \cite{hogadi} showed that a degeneration of a rationally connected variety always contains a rationally connected subvariety. It is not hard then to deduce our transfer principle in equal characteristic $0$ from  Theorem \ref{criterion}. Alternatively, at least in the case $K=\bigcup_{n\in \N}k(\!(t^{1/n})\!)$, one can use the existence of semistable models after ramified base change. We refer the reader to \S \ref{equal0sec}.
%in equal characteristic $0$.
%The general case can be reduced to the above by Ax-Kochen/Ershov in equal characteristic $0$. 

In mixed and positive characteristics, it is not known if degenerations of (separably) rationally connected varieties always contain a (separably) rationally connected subvariety, although there are statements in the literature predicting such phenomena (see Suggestions 7.9 \cite{colliot}).  Nevertheless, for several fields which were previously known to be geometrically $C_1$, it has been shown  that they also admit rational points in degenerations of separably rationally connected varieties. Starr \cite{starr2} calls such fields \textit{RC solving}. For instance, finite fields are RC solving by a theorem of Esnault \cite{esnault2} and Esnault-Xu \cite{esnaultxu}. It is not difficult now to deduce Theorem \ref{hendefdiv1}. We first reduce to the case where $\Oo_K$ is a direct limit of DVRs (in particular rank $1$), using results of Kuhlmann. Then, we conclude using Theorem \ref{criterion}. 
\subsubsection{From geometrically $C_1$ to $C_1$}
%Finally, we explain how to prove that $K$ is $C_1$. 
%Finally, we explain how to deduce the $C_1$ 
%Note that this is not immediately obvious because a singular surface 
If $K$ is geometrically $C_1$ of characteristic $0$, then it is also $C_1$ by Hogadi-Xu \cite{hogadi}. In positive characteristic, it is not known in general if geometrically $C_1$ fields are $C_1$. Fortunately, there is a trick to prove this for fields $K$ as in Theorem \ref{hendefdiv1}. First, note that we can deform any hypersurface $X \subseteq \mathbb{P}^n_K$ of degree $d\leq n$ to a general hypersurface $Y$ of the same degree over $K(\!(t)\!)$. Then $Y$ is separably rationally connected by a result of Zhu \cite{zhu}. Now the key observation is that $L=\bigcup_{n\in \N} K(\!(t^{1/n})\!)$ is geometrically $C_1$ since it is an elementary extension of $K$ by yet another application of Kuhlmann's theory. It follows that $Y$ has an $L$-point. This, in turn, specializes to a $K$-point of $X$ by the valuative criterion of properness. We conclude that $K$ is $C_1$.

%Next, we find some geometrically $C_1$ fields
%, as needed.
%The proof uses a deformation theoretic argument and the fact that $K(\!(t^{\Q})\!)$ is geometrically $C_1$, the latter follows elementary equivalent to $K$ by Kuhlmann
%so some additional work is needed. 
%The key point is that $K\equiv K(\!(t^{\Q})\!)$ in $L_{\text{rings}}$, so that $K$  
\section{Ultralimits of RC varieties}
We will need a fact due to Duesler-Knecht, which roughly says that (separable) rational connectedness behaves well with ultralimits. This statement appears in the first arXiv version \cite{duesler1} but not in the published version \cite{duesler}. We present a slightly stronger version and recall some useful terminology.
%also . 
%Our proof utilizes the uniform definability of Zariski closures in ACF.  
%and utilizes the uniform definability of algebraic closures in ACF. 
%\bp  [cf. Proposition 17 \cite{duesler1}]\label{limitofRC}
%Let $I$ be an index set, $U$ be an ultrafilter on $I$ and $k_i$ be a field for each $i\in I$. Let $k=\prod_{i\in I} k_i/ U$ and $X$ be a projective separably rationally connected $k$-variety. Let $(X_i)_{i\in I}$ be a sequence of $k_i$-schemes with $\ulim X_i=X$. Then the following are equivalent: 
%\begin{enumerate}[label=(\roman*)]
%\item $X$ is a separably rationally connected variety. 

%\item $X_i$ is a separably rationally connected variety for almost all $i$. 

%\end{enumerate}

%\ep  

%saying that rational connectedness behaves well with ultraproducts
%Given a field extension $K/k$, we write $X_{K}=X\times_k K$.
\subsection{Varieties over ultraproducts}
By convention, a \textit{variety} over a field $k$ is a reduced, separated $k$-scheme of finite type.

Let $I$ be an index set, $U$ be an ultrafilter on $I$ and $k=\prod_{i\in I} k_i/ U$ be an ultraproduct of fields. 
\begin{notation}[\S 2 \cite{duesler}]
Let $X\subseteq \mathbb{P}^n_k$ be a closed subscheme  with vanishing homogeneous ideal $I=(f_1,...,f_n)$ where each $f_j(X_0,...,X_n)\in k[X_0,...,X_n]$ is a homogeneous polynomial. For each $i\in I$, let $f_k^{(i)}(X_0,...,X_n)\in k_i[X_0,...,X_n]$ be such that $\ulim f_k^{(i)}=f_k$ and let $X_i \subseteq \mathbb{P}^n_{k_i}$ be the closed subscheme defined by $I^{(i)}$. We write $\ulim X_i=X$ and call $X$ the ultralimit of the $X_i$. 
\end{notation}
A similar definition works also for quasi-projective $k$-schemes. We introduce an analogous notation for morphisms:
\begin{notation}
%[cf. \S 2.4 \cite{Schoutens}]
Let $X\subseteq \mathbb{P}^n_k$ and $Y\subseteq \mathbb{P}^m_k$ be projective schemes and $f:X\to Y$ be a morphism. This gives rise to a (non-unique) sequence $(f_i)_{i\in I}$, where $f_i:\mathbb{P}^n_{k_i}\to \mathbb{P}^m_{k_i}$ is a morphism and $f_i(X_i)\subseteq Y_i$ for almost all $i$. We write $f=\ulim f_i$.
\end{notation}
%Beware that the $X_i$'s are not unique. 
A typical phenomenon with ultralimits is that a given geometric property holds for $X$ if and only if it holds for $X_i$ for \textit{almost all} $i\in I$ (in the sense of the ultrafilter):
\bl \label{ultraprodvar}
Let $I$ be an index set, $U$ be an ultrafilter on $I$ and $k=\prod_{i\in I} k_i/ U$. Let $X$ be a quasi-projective $k$-scheme and suppose that $X=\ulim X_i$. Then: 
\begin{enumerate}[label=(\roman*)]
\item $X$ is reduced (resp. irreducible) if and only if $X_i$ is reduced (resp. irreducible) for almost all $i\in I$.
\item $X$ is smooth if and only if $X_i$ is smooth for almost all $i\in I$.
% In particular, if $X$ is a $k$-variety, then $X_i$ is a $k_i$-variety for almost all $i\in I$. 
%In particular, $X$ is a variety if and only if $X_i$ is a variety for almost all $i$.
\end{enumerate}
\el 
\begin{proof}
%These properties can be checked locally, so assume that $X$ is affine.
For part (i), see \S 2.4.10 \cite{Schoutens}. Part (ii) follows from the Jacobian criterion for smoothness and \L o\'s' theorem. 
\end{proof}
Similarly, if $f=\ulim f_i$, then typically a given geometric property will hold for $f$ if and only if it holds for $f_i$ for \textit{almost all} $i\in I$:
%\bl [Corollary 10.1.6 \cite{johnson}] \label{definabofclosures}
%Let $X_z$ be a definable family of subsets of $\mathbb{P}^n$. Then, the family $\overline{X_z}$ of Zariski closures is definable. That is, given a formula $\phi(x,z)$, there exists a formula $\psi(x,z)$ such that for every $K \models \text{ACF}$ and every $a\in K^m$, the Zariski closure of $X_a(K)=\phi(K,a)$ is equal to $\overline{X_a}(K)=\psi(K,a)$.
%\el 
%\bc \label{definabofclosures}
%Let $X_z \subseteq \mathbb{P}^n_z$ and $Y_z\subseteq \mathbb{P}^m_z$ be definable families (where $z$ is a tuple) in ACF and $F_z:\mathbb{P}^n_z\to \mathbb{P}^m_z$ be a definable family of maps. Then, the set 
%$$D=\{z: \overline{F_z(X_z)}=Y_z\}$$
%is definable in ACF.
%\ec
%For lack of reference, we include a proof of the following:
\bl \label{ultradominant}
Let $I$ be an index set, $U$ be an ultrafilter on $I$ and $k=\prod_{i\in I} k_i/ U$. Let $X$ and $Y$ be quasi-projective $k$-schemes and $f:X\to Y $ be a morphism.
Let $X=\ulim X_i$ and $Y=\ulim Y_i$ and $f=\ulim f_i$. Then: 
\begin{enumerate}[label=(\roman*)]
%\item $\overline{\ulim X_i}=\ulim \overline{X_i}$.
\item $f$ is dominant if and only if $f_i$ is dominant for almost all $i\in I$.
\item Let $K_i$ be a field extension of $k_i$ and $x_i\in X(K_i)$. Let $K=\prod_{i\in I} K_i/U$ and $x=\ulim x_i \in X(K)$. Then $f$ is smooth at $x$ if and only if $f_i$ is smooth at $x_i$ for almost all $i\in I$.
\item Assume that $X$ and $Y$ are integral $k$-varieties and $f$ is dominant. Then $f$ is separable if and only if $f_i$ is separable for almost all $i\in I$.
\end{enumerate}
\el 
\begin{proof}
%It is clear that $\ulim \overline{X_i}$ is a closed subset containing $X=\ulim X_i$, hence $\overline{\ulim X_i} \subseteq \ulim \overline{X_i}$.
%It is convenient to switch to Weil's language of algebraic geometry: 
%Let $K_i$ be an algebraically closed field containing $k_i$ and $K=\prod_{i\in I} K_i/ U$ which is an algebraically closed field containing $k$. \\
(i) %Switching to Weil's language of algebraic geometry, it suffices to show that $f:X(K)\to Y(K)$ is dominant if and only if $f_i:X_i(K_i)\to Y_i(K_i)$ is dominant for almost all $i$. Let $X\subseteq \mathbb{P}^n_k$ and $Y\subseteq \mathbb{P}^m_k$. 
This is essentially a corollary of Johnson's uniform definability of Zariski closures in ACF (see \S 10 \cite{johnson}). Working in ACF, there is a definable family $X_{z}$ (resp. $Y_{z}$) of subsets of $\mathbb{P}^n$ (resp. $\mathbb{P}^m$), a definable family of maps $F_{z}:\mathbb{P}^n\to \mathbb{P}^m$ and parameters $a_i\in k_i^r$ and $a\in k^r$ such that $X_a=X$ and $X_{a_i}=X_i$ (resp. $Y_a=Y$ and $Y_{a_i}=Y_i$) and $F_a=f$ (resp. $F_{a_i}=f_i$). Strictly speaking, an \quotes{equality} such as $X_a=X$ should be interpreted as saying that $X_a(K)=X(K)$ for any algebraically closed field $K$ containing $k$ (similarly for the other equalities).
%the defining equations for $X$ are given by the formula defining $X_a$.
Note that $F_z(X_z)$ is a definable family of subsets of $\mathbb{P}^m$. The same is true for the family of Zariski closures $\overline{F_z(X_z)}$ by Corollary 10.1.6 \cite{johnson}. Therefore, the set 
$$D=\{z: \overline{F_z(X_z)}=Y_z\}$$ 
is definable. By \L o\'s' theorem, we get that $a\in D$ if and only if $a_i\in D$ for almost all $i\in I$. We conclude that $f$ is dominant if and only if $f_i$ is dominant for almost all $i\in I$. \\
%We view $X$ as the fiber of a uniformly definable family $\Xx\subseteq \mathbb{P}^n$ in ACF. 
%$$D=\{a : \overline{\Xx_x}=Y\} $$
%This is a definable set and $a\in D$. By \L o\'s, we get that $a_i \in D$ for almost all $i\in I$.\\
%By Chevalley's theorem, the image $f(X)\subseteq Y$ is constructible, hence $f(X)$ contains a Zariski dense open subset $U\subseteq Y$. Let $U=\ulim U_i$. By \L o\'s , the image $f_i(X_i)$ contains $U_i$, for almost all $i$. It follows that $f_i$ is dominant for almost all $i\in I$.\\
%Let $f_i^*:k_i[Y_i]\to k_i[X_i]$ be the morphism of $k_i$-algebras induced from $f_i$. We also have an induced morphism of $k$-algebras $\ulim f_i^*:\ulim k[Y_i]\to \ulim k[X_i]$. 
%We have $k[X] \subseteq \ulim k_i[X_i]$ and $k[Y] \subseteq \ulim k_i[Y_i]$. 
%\[
  %\begin{tikzcd}
    %\ulim k_i[Y_i]  \arrow[r] & \ulim k_i[X_i] \\
   %[Y]  \arrow[u] \arrow[r, "f^*"] & k[X] \arrow[u] 
    % \end{tikzcd}
%\]
%We have that $f:X\to Y$ is dominant if and only if $f^*:k[Y]\to k[X]$ is  injective. \\
(ii) We use the infinitesimal lifting criterion for smoothness (see Lemma 02H6 \cite{sp}). Suppose that $f_i:X_i\to Y_i$ is smooth for almost all $i$. Suppose we have a commutative diagram 
\[
  \begin{tikzcd}
      \Spec (K)  \arrow[d] \arrow[r, "x"] & X \arrow[d,"f"]\\
  \Spec(K[t]/t^2)   \arrow[r,"g"] & Y
     \end{tikzcd}
\]
For almost all $i$, this gives rise to a commutative diagram 
\[
  \begin{tikzcd}
      \Spec (K_i)  \arrow[d] \arrow[r,"x_i"] & X_i \arrow[d, "f_i"]\\
  \Spec(K_i[t]/t^2)   \arrow[r, "g_i"] \arrow[ur, dotted, "\exists P_i "]  & Y_i
     \end{tikzcd}
\]
Since $f_i$ is smooth at $x_i$, there is $P_i:\Spec(K_i[t]/t^2)\to X_i$ lifting $g_i$ as above. Now $P=\ulim P_i$ gives a well-defined morphism $P:\Spec(K[t]/t^2) \to X $ lifting $g$. 

Conversely, suppose that $f$ is smooth at $x$ and that $f_i$ is not smooth at $x_i$ for almost all $i$. Then there exists $g_i:  \Spec(K_i[t]/t^2) \to Y_i$ which does not lift to $X_i$. Note that $g=\ulim g_i$ gives a well-defined morphism. Since $f$ is smooth, we get that $g$ lifts to $P:\Spec(K[t]/t^2)\to X$. Write $P=\ulim P_i$ and note that $P_i$ gives a lift of $g_i$ for almost all $i$, a contradiction. \\
%$\Spec(K[t]/t^2)\to X$. \\
(iii) For almost all $i$, $X_i$ and $Y_i$ are integral varieties by Lemma \ref{ultraprodvar}(i) and $f_i$ is dominant for almost all $i$ by (i). Recall that a dominant morphism of integral varieties is separable if and only if it is generically smooth. For almost all $i$, let $K_i=k_i(X_i)$ and $K=\prod_{i\in I}K_i/U$. Let $x_i\in X_i(K_i)$ be the generic point of $X_i$. We claim that $x=\ulim x_i\in X(K)$ is above the generic point of $X$. Indeed, let $U\subseteq X$ is any nonempty Zariski open subset. We can write $U=\ulim U_i$, where $U_i\subseteq X_i$ is a Zariski open subset which is nonempty for almost all $i$. Since $x_i\in U_i(K_i)$, we get that $x\in U(K)$. It follows that $x$ is above the generic point of $X$. 
By (ii), we have that $f$ is smooth at $x$ if and only if $f_i$ is smooth at $x_i$ for almost all $i$. The conclusion follows.
%It suffices to show that $f:X(K)\to Y(K)$ is smooth at the generic point if and only if the same is true for $f_i:X_i(K_i)\to Y_i(K_i)$ for almost all $i$. This follows from the Jacobian criterion for smoothness and \L o\'s' theorem. 
%we see that the map $f_i:V_i\to Y_i$ is smooth for almost all $i\in I$. It follows that the map $f_i:X_i\to Y_i$ is separable for almost all $i$. 
%Suppose that $f:X\to Y$ is separable. Then, there exists a non-empty Zariski open subset $V\subseteq X$ such that $f:V \to Y$ is smooth. Write $V=\ulim V_i$, where $V_i\subseteq X_i$. By Lemma \ref{ultraprodvar}, we get that $f_i:V_i\to Y_i$ is smooth for almost all $i$, hence $f_i:X_i\to Y_i$ is separable for almost all $i$. 
%PROPOSED PROOF: SEPARABLE IF AND ONLY IF THERE EXISTS A SMOOTH FIBER.
%Using the Jacobian criterion and \L o\'s, we see that the map $f_i:V_i\to Y_i$ is smooth for almost all $i\in I$. It follows that the map $f_i:X_i\to Y_i$ is separable for almost all $i$. 
%CONVERSE (THOUGH NOT NEEDED.)
%Assume that $f$ is dominant. We then have
\end{proof}

\subsection{Rational connectedness and ultralimits}
%Duesler-Knecht
Part (i) below is stated in Proposition 17 \cite{duesler1} only in the smooth projective case. Moreover, while it is not explicitly stated in Proposition 17 \textit{ibid}, it is assumed in the proof that almost all the $X_i$ have the same Hilbert polynomial. Part (ii) is as in \cite{duesler1} and  is only stated for completeness, it will not be used anywhere in the paper. 
%Note also that Duesler-Knecht \cite{duesler1} assume that almost all $X_i$ have the same Hilbert polynomial but 
%The proof is also substantially different.
%We prove the statement holds in general. 
\bp  [cf. Proposition 17 \cite{duesler1}]\label{limitofRC}
Let $I$ be an index set, $U$ be an ultrafilter on $I$ and $k_i$ be a field for each $i\in I$. Let $k=\prod_{i\in I} k_i/ U$ and $X$ be a quasi-projective separably rationally connected $k$-variety. Let $(X_i)_{i\in I}$ be a sequence of $k_i$-schemes with $\ulim X_i=X$. Then: 
\begin{enumerate}[label=(\roman*)]
\item If $X$ is separably rationally connected, then $X_i$ is separably rationally connected for almost all $i$. 

\item If almost all $X_i$ have the same Hilbert polynomial and are smooth projective and separably rationally connected, then the same is true for $X$. 

\end{enumerate}

\ep  
\begin{proof}
%[Proof of Proposition \ref{limitofRC}]
%We can assume that each $k_i$ is algebraically closed (hence also $k$). 
(i) By Lemma \ref{ultraprodvar}, $X_i$ is a variety for almost all $i\in I$.
By assumption, there exists a $k$-variety $B$ and a map $F:B\times \mathbb{P}^1\to X$ such that
$$G:B\times \mathbb{P}^1\times \mathbb{P}^1 \to  X\times X: (b,t,t')\mapsto (F(b,t),F(b,t'))$$
is dominant and separable. After replacing $B$ with an open subscheme, if needed, 
%POSSIBLY DEBARRE'S DEFINITION B\to X\times X is dominant and separable.
%Nagata's compactification theorem, $B$ embeds into a and then compactifying $B$, we get a projective $k$-variety $B'$ birational to $B$ and rational maps $F'$ and $G'$ as above. By blowing up, we can resolve indeterminacies and make $F'$ (and $G'$) 
we can assume that $B$ is a quasi-projective variety.  
%Namely, work with Weil universe and replace B with an irreducible component and then an open dense subscheme.
%Namely, let $\eta$ be the generic point of $X\times X$ and $\eta'\in B\times \mathbb{P}^1\times \mathbb{P}^1$ be such that $G(\eta')=\eta$. Let $\eta_B=\pi_1(\eta')$.
Write $B=\ulim B_i$ and $F=\ulim F_i$. Then, we have $G=\ulim G_i$, where 
$$G_i:B_i\times \mathbb{P}^1\times \mathbb{P}^1 \to X_i\times X_i: (b,t,t')\mapsto (F_i(b,t),F_i(b,t')) $$ 
is well-defined for almost all $i\in I$. By Lemma \ref{ultradominant}, we get that $G_i$ is dominant and separable for almost all $i$. We conclude that $X_i$ is separably rationally connected for almost all $i\in I$.\\
(ii) 
%See Proposition 17 \cite{duesler1} or the proof of Theorem 11 \cite{duesler}.
%
In Proposition 17 \cite{duesler1}, Duesler-Knecht prove that $X$ is separably rationally connected by showing the existence of a very free rational curve on $X$. Alternatively, one can use the argument in the proof of Theorem 11 \cite{duesler}. Although the authors assume characteristic $0$, their argument ultimately relies on IV, Theorem 3.11 \cite{kollarbook} which is valid for separable rational connectedness and in arbitrary characteristic. 
\end{proof}
%We point out that Proposition \ref{limitofRC} has an easy proof in case $X$ is smooth and almost all $k_i$ are of character zero:
%\begin{proof}
%Since $X$ is smooth, the same is true for almost all $X_i$.  Suppose that $X_i$ is not rationally connected. Then, for almost all $i$, there are $x_i,y_i\in X_i(\overline{k_i})$, such that $x_i$ and $y_i$ are not connected via a rational curve. Let $x^*=\ulim x_i$, $y^*=\ulim y_i$ and $\overline{k}^*= \prod_{i\in I} \overline{k_i}/U$. Note that $x^*,y^* \in X(\overline{k}^*)$. Since $X$ is rationally connected, we get that $X_{\overline{k}^*}$ is rationally connected. Let $f:\mathbb{P}^1_{ \overline{k}^*} \to X(\overline{k}^*)$ be a rational curve connecting $x^*$ and $y^*$. 
%\end{proof}
%\begin{fact}[ IV. 3.3 \cite{kollarbook}]
%If $\text{char}(k)=0$, then a $k$-variety is separably rationally connected if and only if it is rationally connected.
%\end{fact}
%A crucial fact is that (separable) rational connectedness is deformation invariant in smooth families:
%\bl [Theorem 3.11 \cite{kollarbook}] \label{opencondition}
%Let $X\to S$ be a smooth morphism of schemes with $S$ connected. Write $X_s=X\times_S s$ for each $s\in S$. Then the set
%$$U=\{s\in S: X_s \mbox{ is separably rationally connected}\} $$
%is open in $X$.
%\el 

\section{Variations on the $C_1$ property} \label{c1fieldsandvariants}
%Given a language $L$, 
A class $\mathcal{C}$  of $L$-structures is called an \textit{elementary class} if it is the class of models of some $L$-theory $T$. The class of $C_1$ fields is clearly an elementary class which is $\forall \exists$-axiomatizable in $L_{\text{rings}}$. We prove that the same is true for the classes of geometrically $C_1$ fields and RC solving fields. 
%We also introduce the class of strongly
%form elementary classes which are $\forall \exists$-axiomatizable in $L_{\text{rings}}=\{+,\cdot,0,1\}$. 
%We develop a systematic model-theoretic study of these classes, thus putting the results of \cite{duesler} in a larger context.
%To some extent, the former fact is already implicit in . 
%The former fact is arguably implicit in \cite{duesler1}.
% and can even be axiomatized  by $\forall \exists$-sentences.
\subsection{Geometrically $C_1$ fields}
%\begin{definition}[Koll\'ar]
%A field $k$ is called geometrically $C_1$ if every smooth projective separably rationally connected variety over $k$ has a $k$-rational point.
%\end{definition}
% Over fields of characteristic $0$, one can assume projective.
Following a suggestion of Koll\'ar, Hogadi-Xu define geometrically $C_1$ fields in characteristic $0$ (see Definition 1.4 \cite{hogadi}). Their definition extends naturally in arbitrary characteristic:
\begin{definition}
A field $k$ is called \textit{geometrically }$C_1$ if every smooth projective separably rationally connected $k$-variety has a $k$-rational point. 
\end{definition}
\begin{rem}
\begin{enumerate}[label=(\roman*)]

\item Hogadi-Xu use proper varieties rather than projective varieties but this does not make a difference in view of Chow's lemma.

\item If $k$ is geometrically $C_1$ and large of characteristic $0$, then every rationally connected $k$-variety (not necessarily smooth or proper) has a $k$-point (see Lemma 3.2 \cite{pieropan}). 
%dense many $k$-rational points.
\end{enumerate}
\end{rem}
%\bl 
%Let $k$ be a geometrically $C_1$ field of characteristic $0$. Then: 
%\begin{enumerate}[label=(\roman*)]
%\item Every smooth proper rationally connected $k$-variety has a $k$-rational point. 
%\item 
%\end{enumerate}
%\el 
%\begin{proof}
%See Lemmas 3.1, 3.2 \cite{pieropan}.
%\end{proof}
%For large fields of characteristic $0$, one can drop smooth and proper:
%\bl [Lemma 3.2 \cite{pieropan}] \label{largelemma}
%Let $k$ be a large field of characteristic $0$ which is geometrically $C_1$. Then, every rationally connected variety has Zariski dense many $k$-rational points.
%\el
%\begin{proof}
%Note that it suffices to show that every smooth rationally connected variety over $K$ has Zariski dense $K$-rational points. 
%Let $X$ be an arbitrary rationally connected variety. By Hironaka's resolution of singularities \cite{hironaka}, we find a proper birational morphism $X'\to X$ with $X'$ smooth and proper. Since $X$ is rationally connected and $X$ is birational to $X'$, we get that $X'$ is rationally connected. By assumption, we have $X'(k)\neq \emptyset$. Since $X'$ is smooth and $k$ is large, we get that $X'(k)$ is Zariski dense in $X'$. It follows that $X(k)$ is also Zariski dense in $X$.
%\end{proof}

\bl \label{easyprops}
Let $I$ be an index set and $U$ be an ultrafilter on $I$. 
\begin{enumerate}[label=(\roman*)]

%\item Let $I$ be an directed set and $k_i$ be a geometrically $C_1$ field for each $i\in I$. Then $\varinjlim k_i$ is geometrically $C_1$. 
\item Let $k_i$ be a field which is geometrically $C_1$, for almost all $i\in I$. Then $\prod_{i\in I} k_i/ U$ is geometrically $C_1$. 

\item Let $k$ be a field. Then, $k$ is geometrically $C_1$ if and only if  $k_U$ is geometrically $C_1$.

\item Suppose that $k\equiv l$. Then $k$ is geometrically $C_1$ if and only if $l$ is geometrically $C_1$.

\end{enumerate}
\el 
\begin{proof}
%(i) Clear.\\
(i) Let $k=\prod_{i\in I} k_i/ U$ and $X\subseteq \mathbb{P}^n_k$ be a smooth projective separably rationally connected $k$-variety. Let $(X_i)_{i\in I}$ be a sequence of projective $k_i$-schemes with $\ulim X_i=X$. By Lemma \ref{ultraprodvar} and Proposition \ref{limitofRC},  for almost all $i$, the scheme $X_i$ is a smooth projective separably rationally connected $k_i$-variety. By assumption, we get that $X_i(k_i)\neq \emptyset$ for almost all $i$. By \L o\'s, it follows that $X(k)\neq \emptyset$. \\
(ii) The \quotes{only if} follows from (i). For the converse, let $X$ be a smooth projective separably rationally connected $k$-variety. By assumption, we get that $X(k_U)\neq \emptyset$. Since $k\preceq_{\exists} k_U$ in the language of rings, it follows that $X(k)\neq \emptyset$.\\
%Then, the base change $X_{k_U}=X\times_k k_U$ is a smooth projective separably rationally connected $k_U$-variety and hence \\
(iii) By the Keisler-Shelah theorem, we have $k_U\cong l_U$ for some ultrafilter $U$. We conclude from (ii).
%By (i), we get that $k_U$ is geometrically $C_1$. By (ii), we get that $l$ is geometrically $C_1$.\\
\end{proof}

\bl \label{easyprops2}

\begin{enumerate}[label=(\roman*)]
\item Let $I$ be a directed system and $k_i$ be a geometrically $C_1$ field for each $i\in I$. Then $\varinjlim k_i$ is geometrically $C_1$. 

\item Let $k$ be a geometrically $C_1$ field and $l/k$ be a separable algebraic extension. Then $l$ is geometrically $C_1$.

\item Let $(K,v)$ be a valued field of equal characteristic with residue field $k$. If $K$ is geometrically $C_1$, then $k$ is geometrically $C_1$.
\end{enumerate}
\el 
\begin{proof}
(i) Let $X$ be a smooth projective separably rationally connected $k$-variety. Then there is $i\in I$ and a $k_i$-scheme $X_i$ such that $X=X_i\times_{k_i} k$. Moreover, $X_i$ is a smooth projective separably rationally connected $k_i$-variety. By assumption, we have $X_i(k_i)\neq \emptyset$, hence $X(k)\neq \emptyset$.\\
(ii) By part (i), it suffices to prove the statement when $l/k$ is a finite separable extension. 
%We now use an
%d\'evissage 
%argument due to de Jong (see \S 3.2 \cite{grabberstarr}). 
Let $X$ be a smooth projective separably rationally connected $l$-variety. Then, the Weil restriction $\text{Res}_{l/k}(X)$ is a smooth, projective separably rationally  connected $k$-variety. Indeed, the base change $\text{Res}_{l/k}(X)\times_k l$ is a finite product of copies of $X$ and, as such, it is smooth, projective and separably rationally connected (cf. \S 3.2 \cite{grabberstarr}). 
%A product of SRC is SRC: X\times X 
By assumption, we have that $\text{Res}_{l/k}(X)(k)\neq \emptyset$, hence $X(l)\neq \emptyset$.\\
%We have a $k$-embedding $l\to k_U$, for some ultrafilter $U$. Let $V$ be a rationally connected $l$-variety. By Fact \ref{geomprop}, the variety $V_{k_U}=V\times_l k_U$ is also rationally connected. By (ii), we know that $V_{k_U}$ has Zariski dense many $k_U$-rational points. Since 
(iii) By part (ii), the henselization $K^h$ of $K$ is also geometrically $C_1$, so we can assume that $K$ is henselian. Now let $X$ be a smooth projective separably rationally connected $k$-variety. Let $k_0\subseteq k$ be the field of definition of $X$, which is finitely generated over its prime subfield. Then, there exists a lift $\iota: k_0 \to \Oo_K$, namely a ring homomorphism such that $\text{res}(\iota(x))=x$ for all $x\in k_0$. One proves the existence of such a map by first lifting a separable transcendence basis of $k_0$ over its prime subfield and then applying Hensel's Lemma (cf. \S 2.4 \cite{vdd}).  
Using $\iota$, we obtain a flat, projective $\Oo_K$-scheme $\Xx$ with special fiber $X$, namely the trivial deformation over $\Oo_K$. The generic fiber is equal to $X\times_{k_0} K$ and hence is a smooth, projective separably rationally connected $K$-variety. Since $K$ is geometrically $C_1$, there exists $P_K\in \Xx(K)$. Since $\Xx$ is proper, this extends to an integral point $P\in \Xx(\Oo_K)$. By reduction modulo $\mathfrak{m}$, we get that $P_k\in X(k)$.
%Let $X$ be a smooth projective separably rationally connected $k$-variety. Since $X$ is smooth, there is an affine open $U\subseteq X$ of the form $U=\Spec(A)$, where $k[T_1,...,T_n]\to A$ is \'etale. This lifts to an \'etale map $\Oo_K[T_1,...,T_n]\to \tilde{A}$. Set $\tilde{U}=\Spec(\tilde{A})$. By Theorem 3.11 \cite{kollarbook}, we get that the generic fiber $\tilde{U}_K$ is separably rationally connected.  \\
%and 
\end{proof}

%\bq 
%Is part (ii) also true in mixed characteristic?
%\eq 
\bp \label{geomc1elemclass}
The class of geometrically $C_1$ fields is elementary in $L_{\text{rings}}$. Moreover, it is $\forall \exists$-axiomatizable.
\ep   
\begin{proof}
The first part follows from Lemma \ref{easyprops}(i), (iii) and Theorem 4.1.12 \cite{changkeisler}. The moreover part follows from Lemma \ref{easyprops2}(i) and Corollary 3.1.9 \cite{tent}.
%It suffices to show that it is closed under ultraproducts and elementary equivalence. Finally, we verify that our class is closed under elementary equivalence. Let $k$ be geometrically $C_1$ field, $l \equiv k$ and $V$ be a rationally connected variety over $l$. For some ultrafilter $U$, we have $l\preceq_{\exists} k_U$. By the previous paragraph, we get that $V(k_U)$ is Zariski dense in $V$ and therefore the same is true for $V(l)$.
\end{proof}
%We now discuss the relation with $C_1$ fields.
%\begin{fact} [Theorem 1.2 \cite{hogadi}] \label{hogadi}
%Let $k$ be a field of characteristic $0$ and $R$ be a DVR with residue field $k$. Let $X$ be a rationally connected $K$-variety and $\Xx$ be an $R$-model. Then $\Xx_k$ contains a rationally connected $k$-subvariety.
%\end{fact}
%As an immediate consequence:

\subsection{RC solving fields}
Following Starr \S 3 \cite{starr2}, we also consider fields which admit rational points on \textit{degenerations} of rationally connected varieties. In order to make some of the proofs work, Starr needs to assume that the base DVR of the degeneration is a \textit{prime regular} DVR (defined below). This   assumption is only needed for mixed characteristic DVRs.
\begin{definition}
%[Definition 3.3 \cite{starr2}]
A DVR $(\Lambda,\mathfrak{m}_{\Lambda})$ is called \textit{prime finite} if the residue field $\Lambda/\mathfrak{m}_{\Lambda}$ is a finite extension of its prime subfield.
\end{definition}

\begin{definition} 
%[\S 3 \cite{starr2}]
Let $\phi: (\Lambda, \mathfrak{m}_{\Lambda})\to (R,\mathfrak{m}_R)$ be a morphism of DVRs. We say that $\phi$ is \textit{regular} if the following conditions hold: 
\begin{enumerate}[label=(\roman*)]

\item We have $\mathfrak{m}_R=\phi(\mathfrak{m}_{\Lambda})R$, i.e., $\phi$ is weakly unramified.

\item The extension $k/\ell$ is separable, where $k=R/\mathfrak{m}_R$ and $\ell=\Lambda/\mathfrak{m}_{\Lambda} $.

\item The extension $K/L$ is separable, where $K=\text{Frac}(R)$ and $L=\text{Frac}(\Lambda)$.

\end{enumerate}
\end{definition}

\begin{definition} \label{primeregdefn}
%[\S 3 \cite{starr2}] 
A DVR $(R,\mathfrak{m}_R)$ is called \textit{prime regular} if there exists a prime finite DVR $(\Lambda, \mathfrak{m}_{\Lambda})$ and a regular morphism $\phi:(\Lambda,\mathfrak{m}_{\Lambda})\to (R,\mathfrak{m}_R)$.
\end{definition}

\begin{definition}
%[Definition 3.8 \cite{starr2}]
Let $k$ be a field.
\begin{enumerate}[label=(\roman*)]
\item We say that $k$ is \textit{RC solving} if for every prime regular DVR $(R,\mathfrak{m}_R)$ with $R/\mathfrak{m}_R\subseteq k$ and $K=\text{Frac}(R)$ and for every flat, projective $R$-scheme $\Xx$ with $\Xx_K$ smooth and separably rationally connected, we have that $\Xx(k)\neq \emptyset$. 

\item We say that $k$ is  \textit{equal characteristic RC solving} if the above condition holds for DVRs of equal characteristic.
\end{enumerate}
\end{definition}
%Later, 
We note that the terminology \quotes{equal characteristic RC solving} is not used in \cite{starr2}. 
\bl \label{rcimpliesgeomc1}
Let $k$ be an equal characteristic RC solving field. Then: 

\begin{enumerate}[label=(\roman*)]
\item $k$ is geometrically $C_1$.
\item $k$ is $C_1$.
\end{enumerate}
\el 
\begin{proof}
(i) Let $X$ be a smooth projective separably rationally connected $k$-variety and consider the trivial deformation $\Xx\to \Spec(R)$ of $X$ over $R=k[\![t]\!]$. Since $k$ is equal characteristic RC solving, we get that $\Xx_k(k)\neq \emptyset$ and hence $X(k)\neq \emptyset$.\\
%\end{proof}
%\bl \label{rcimpliesc1}
%Let $k$ be an equal characteristic RC solving field. Then 
%\el 
%\begin{proof}
(ii) Let $X\subseteq \mathbb{P}^n_k$ be a hypersurface of degree $d\leq n$. Let $M$ be the parameter space of hypersurfaces of degree $d\leq n$ in $\mathbb{P}^n$ and $\Xx_M\to M$ be the universal family. Recall that $M$ is a smooth $\Z$-scheme isomorphic to $\mathbb{P}^N_{\Z}$ with $N={{n+d-1}\choose{n-1}}$. Let $\zeta_k: \Spec(k)\to M$ be such that $X\to \Spec(k)$ is the pullback of $\Xx_M\to M$ via $\zeta_k$. Since $M$ is a smooth $\Z$-scheme, there is $\zeta:\Spec(k[\![t]\!])\to M$ extending $\zeta_k$ and such that the image of the generic point lies in any specified Zariski open subset of $M$. Therefore, there exists a $k[\![t]\!]$-scheme $\Xx$ such that $\Xx_k=X$ and such that the generic fiber $\Xx_{k(\!(t)\!)}$ is a general hypersurface in $\mathbb{P}^n_{k(\!(t)\!)}$ of degree $d\leq n$. 
%Namely, one can lift the coefficients of the defining equation for $X$ such that they are algebraically independent.
Then, $\Xx_{k(\!(t)\!)}$ is separably rationally connected by \cite{zhu}. Since $k$ is equal characteristic RC solving, we conclude that $X(k)\neq \emptyset$.
\end{proof}

\bl [Corollary 1.5 \cite{hogadi}]\label{hogadi}
Any geometrically $C_1$ field of characteristic $0$ is RC solving. In particular, it is $C_1$.
\el 
\begin{proof}
Although Corollary 1.5 \cite{hogadi} refers only to the $C_1$ property, its proof directly generalizes: Let $k$ be a geometrically $C_1$ field of characteristic $0$. Let $R$ be a DVR with residue field $k$ and $K=\text{Frac}(R)$. Let $\Xx$ be a flat, projective $R$-scheme $\Xx$ with $\Xx_K$ smooth and rationally connected. 
%We can assume that $R$ is complete, so $R=k[\![t]\!]$. 
%For each $n\in \N$, let $R_n=k[\![t^{1/n}]\!]$ and $K_n=\text{Frac}(R_n)$. Let $\Xx$ be a proper, flat $R$-scheme with $\Xx_K$ smooth and rationally connected. 
By Theorem 1.2 \cite{hogadi}, % (see also Th\'eor\`eme 7.15 \cite{colliot}), 
there is a $k$-subvariety $Z\subseteq \Xx_k$ which is rationally connected. Then, $Z$ has a $k$-rational point since its resolution $Z'$ is smooth, projective and rationally connected and has a $k$-rational point by our
assumption on $k$. It follows that $\Xx(k)\neq \emptyset$ and hence $k$ is RC solving.
%By resolution of singularities in characteristic $0$, there exists a proper birational morphism $f:\Xx'\to \Xx$ such that $\Xx'$ is a normal crossings $R$-model of $X$. By further blowing up $\Xx'$ if necessary, we can assume that $\Xx'$ is a strict normal crossings model (see \S 1.2.4 \cite{KK2}).
%(cf. Theorem 0.1 \cite{equivariantres}), 
%Since $f$ is a birational morphism, the generic fiber $\Xx'_K$ is rationally connected. 
%By Remark 5.2.4 \cite{brownrational}, there exists an irreducible component $Y_i$ of $\Xx'_k$ such that $Y_i$ is a smooth rationally connected $k$-variety. Since $k$ is geometrically $C_1$, there is $x\in Y_i(k)$. It follows that $f(x)\in \Xx(k)$. 
%The last statement follows from Lemma \ref{rcimpliesgeomc1}(ii).
%in the proof of .
\end{proof}

\bc \label{hogadicor}
Let $k$ be a field of characteristic $0$. Then $k$ is geometrically $C_1$ field if and only if $k$ is RC solving.
%Let $k$ be field of characteristic $0$ such that every rationally connected $k$-variety has a $k$-rational point. Let $X_k$ be a $k$-scheme which is the special fiber of a proper, flat scheme $X$ over a DVR whose generic fiber is rationally connected. Then $X_k$ has a $k$-point.
\ec 
\begin{proof}
From Lemma \ref{hogadi} and Lemma \ref{rcimpliesgeomc1}.
\end{proof}

\bl \label{easylemrc}
Let $I$ be an index set and $U$ be an ultrafilter on $I$. 
\begin{enumerate}[label=(\roman*)]

%\item Let $I$ be an directed set and $k_i$ be a geometrically $C_1$ field for each $i\in I$. Then $\varinjlim k_i$ is geometrically $C_1$. 
\item Let $k_i$ be a field which is RC solving, for almost all $i\in I$. Then $\prod_{i\in I} k_i/ U$ is RC solving. 

\item Let $k$ be a field. Then, $k$ is RC solving if and only if $k_U$ is RC solving.

\item Suppose that $k\equiv l$. Then, $k$ is RC solving if and only if $l$ is RC solving.
\end{enumerate}
\el 
\begin{proof}
%(i) Clear.\\
(i) Set $k=\prod_{i\in I} k_i/ U$ and let $(R,\mathfrak{m})$ be a prime regular DVR with $R/\mathfrak{m} \subseteq k$ and $K=\text{Frac}(R)$. Let $\Xx$ be a flat, projective $R$-scheme with $\Xx_{K}$ smooth and separably rationally connected. Let $(\Lambda,\mathfrak{m}_{\Lambda})$ be a prime finite DVR such that $(\Lambda,\mathfrak{m}_{\Lambda})\to (R,\mathfrak{m}_R)$ is regular. By Lemma 3.2 \cite{starr2}, there exists a smooth parameter space $M$ over $S=\Spec(\Lambda)$, a finite type, flat, projective $M$-scheme $\Xx_M$ and a dominant $S$-morphism $\zeta: \Spec(R)\to M$ such that we have a pullback diagram 
\[
  \begin{tikzcd}
    \Xx\arrow[d] \arrow[r] & \Xx_{M}  \arrow[d,"f"]\\
   \Spec (R)   \arrow[r, "\zeta"] & M
     \end{tikzcd}
\]
Let $M_{\eta}= M\times_{\Spec(\Lambda)} \Spec(\text{Frac}(\Lambda))$ be the generic fiber.
% and $M_0=M\times_{\Spec(\Lambda)} \Spec(\Lambda/\mathfrak{m}_{\Lambda})$ be the special fiber.
By IV, Theorem 3.11 \cite{kollarbook}, there is a Zariski open subset $V\subseteq M_{\eta}$ containing $\zeta(\Spec(K))$ such that $\Xx_u$ is a smooth projective separably rationally connected variety for all $u\in V$. Since $\zeta$ is dominant, the subset $V$ is Zariski dense in $M_{\eta}$. 

Let $x=\zeta(\mathfrak{m}_R)\in M(k)$. In order to get $k_i$-points in $M$,
we need a \quotes{spreading out} argument  for the case where $\text{char}(k)=0$ and almost all $k_i$'s have positive characteristic: Since $M$ is a finite type $\Lambda$-scheme, it is defined over a finitely generated $\Z$-algebra $A\subseteq \Lambda$. We henceforth view $M$ as an $A$-scheme. Since $M$ is of finite type over $A$ (hence over $\Z$), we can apply \L o\'s' theorem to  get a sequence of points $x_i \in M(k_i)$ (unique up to $U$-equivalence). We note that $\Xx_M\times_M \Spec(k)=\ulim \Xx_M \times_M \Spec(k_i)$.
%, we get that $M(k_i Let  be  such that over $A$. We now proceed as before but with $A$ in place of $\Lambda$. Set $S=\Spec(A)$ and note that $S(k_i)\neq \emptyset$ for almost all $i$.\\

%In the following paragraph, we assume that almost all the $k_i$'s have the same characteristic. Since $\Lambda/\mathfrak{m}_{\Lambda}$ is finite over a prime field, we get embeddings $\Lambda/\mathfrak{m}_{\Lambda}\subseteq k_i$ for almost all $i$. In the last paragraph, we explain the necessary changes for the general case. 

%Since $M(k)\neq \emptyset$, there exists $x_i\in M(k_i)$ for almost all $i$. 
%Suppose that $x_i\in M_0(k_i)$ for some $i\in I$ such that $k_i$ is RC solving. 
Let $(R_i,\mathfrak{m}_{i})$ be a henselian prime regular DVR with $R_i/\mathfrak{m}_{i} = k_i$ and $K_i=\text{Frac}(R_i)$. By the implicit function theorem for henselian fields (see Theorem 9.2 \cite{greenrumely}), there is $\zeta_i: \Spec(R_i)\to M$ which specializes to $x_i$ such that $\zeta_i(\Spec(K_i))\in V$. 
%For each $i$, let $(R_i,\mathfrak{m}_{i})$ be a complete prime regular DVR with $R_i/\mathfrak{m}_{i}=k_i$ and $K_i=\text{Frac}(R_i)$. Then, for any $x_i\in M(k_i)$, there exists $P_i\in M(R_i)$ by Hensel's Lemma. Moreover, since $M\to \Spec(\Z)$ is smooth, there is $\zeta_i: \Spec(R_i)\to M$ which specializes to $x_i$ and with $\zeta_i(\Spec(K_i))\in U$. 
Therefore, $\Xx_M\times_M \Spec(K_i)$ is a smooth projective separably rationally connected $K_i$-variety. Since $k_i$ is RC solving, we get that $\Xx_{M}\times_{M} \Spec(k_i)$ has a $k_i$-rational point. 
%We conclude that, for every $x_i\in M_0(k_i)$, the $k_i$-scheme $\Xx_{M_0}\times_{M_0} \Spec(k_i)$ has a $k_i$-rational point. 
Since $\Xx_M\times_M \Spec(k)=\ulim \Xx_M \times_M \Spec(k_i)$, we get that $\Xx(k) \neq \emptyset$.\\
%This is now encoded by a $\forall \exists$-sentence $\phi$ in $L_{\text{rings}}$, namely by writing down equations for the $\Lambda/\mathfrak{m}_{\Lambda}$-scheme $M_0$ and the $M_0$-scheme $\Xx_{M_0}$. Since $\Lambda/\mathfrak{m}_{\Lambda}$ is a finite extension of a prime field, this is indeed expressible in $L_{\text{rings}}$, i.e., without parameters. Therefore, $\phi$ also holds in $k$ by \L o\'s. In particular, we get $\Xx(k)\neq \emptyset$.
%with $S=\Spec(A)$ and proceed as before. Since $S(k)\neq \emptyset$, we have that $S(k_i)\neq \emptyset$ for almost all $i$. \\
%also $\mathfrak{p}=\mathfrak{m}_{\Lambda}\cap A$. Then $A/\mathfrak{p}$ is a finitely generated subring of the number field $\Lambda/\mathfrak{m}_{\Lambda}$.  \\
%Then $\Lambda/\mathfrak{m}_{\Lambda}$ is a number field. \\
%By \L o\'s, we there exist $x_i \in M(k_i)$ for almost all $i$. Let $R_i$ be a henselian prime regular DVR with residue field $k_i$. Since $M$ is smooth, there are $P_i\in M(R_i)$ such that $P_i \in U$. In particular, we get that $\Xx_M \times x_i(k_i)\neq \emptyset$ for almost all $i$.   
(ii) The \quotes{only if} follows from (i). For the converse, let $(R,\mathfrak{m}_R)$ be a prime regular DVR with $R/\mathfrak{m}_R \subseteq k$ and $K=\text{Frac}(R)$ and $\Xx$ be a flat, projective $R$-scheme with $\Xx_{K}$ smooth and separably rationally connected. 
%We can now construct a DVR $(S,\mathfrak{m}_S)$ with $S/\mathfrak{m}_S=k_U$ and a regular morphism $(R,\mathfrak{m}_R)\to (S,\mathfrak{m}_S)$. First, since $k_U/k$ is separable, we can lift a separable transcendence basis $\tau$ to form a regular morphism $(R,\mathfrak{m}_R)\to (R',\mathfrak{m}_{R'})$ with $R'/\mathfrak{m}_{R'}=k(\tau)$. There is now an ind-\'etale (hence regular) morphism $(R', \mathfrak{m}_{R'})\to (S,\mathfrak{m}_S)$ with $S/\mathfrak{m}_S=k_U$. The composite morphism $(R,\mathfrak{m}_R)\to (S,\mathfrak{m}_S)$ is then regular, hence $(S,\mathfrak{m}_S)$
%The morphism $(R,\mathfrak{m}_R)\to (S,\mathfrak{m}_S)$ is regular. 
%is a prime regular DVR. Set $L=\text{Frac}(S)$. Now $\Xx_S=\Xx\times_R S$ is a projective, flat $S$-scheme and $\Xx_L=\Xx_K\times_K L$ is smooth and separably rationally connected. 
Since $k_U$ is RC solving, we get that $\Xx(k_U)\neq \emptyset$. Since $k\preceq_{\exists} k_U$, we conclude that $\Xx(k)\neq \emptyset$.\\
(iii) As in Lemma \ref{easyprops}(iii).
\end{proof}

\bl \label{easylemrc2}

\begin{enumerate}[label=(\roman*)]
\item Let $I$ be a directed system and $k_i$ be a RC solving field for each $i\in I$. Then $\varinjlim k_i$ is RC solving.

\item Let $(K,v)$ be a valued field of equal characteristic with perfect residue field $k$. If $K$ is geometrically $C_1$, then $k$ is equal characteristic RC solving.
\end{enumerate}
\el 
\begin{proof}
(i) Let $(R,\mathfrak{m})$ be a prime regular DVR with $R/\mathfrak{m}=k$ and $K=\text{Frac}(R)$. Let $\Xx$ be a flat, projective $R$-scheme with $\Xx_{K}$ smooth and separably rationally connected. We proceed as in Lemma \ref{easylemrc}(i). Let $M$ be a smooth parameter space over $S=\Spec(\Lambda)$ and $\zeta: \Spec(R)\to M$ be such that $\Xx\to \Spec(R)$ is the pullback of $\Xx_M\to M$ via $\zeta$. Let $V\subseteq M_{\eta}$ be a Zariski dense open subset containing $\zeta(\Spec(K))$ such that $\Xx_u$ is smooth and separably rationally connected for all $u\in V$. Since $M$ is a finite type $\Lambda$-scheme, the morphism $\zeta_0:\Spec(k)\to M$ factors as 
$$\Spec(k)\to \Spec(k_i) \stackrel{\zeta_{0,i}}\to M$$ 
for some $i\in I$. Let $R_i$ be a complete prime regular DVR with residue field $k_i$ and $K_i=\text{Frac}(R_i)$ of the same characteristic as $K$. Since $M$ is smooth, there exists $\zeta_i: \Spec(R_i)\to M$ extending $\zeta_{0,i}$ such that $\zeta_i(K_i)\in V$.
%Let $M_{\eta}= M\times_{\Spec(\Lambda)} \Spec(\text{Frac}(\Lambda))$.
%By Theorem 3.11 \cite{kollarbook}, there is an open subset 
%Since $M(k)\neq \emptyset$, there exists $x_i\in M(k_i)$ for almost all $i$. 
%Suppose $x_i\in M(k_i)$ for some $i\in I$.
%Then, there is a prime regular DVR $(R_i,\mathfrak{m}_{i})$ with $R_i/\mathfrak{m}_{i} \subseteq k_i$ and $K_i=\text{Frac}(R_i)$, and $\zeta_i: \Spec(R_i)\to M$ which specializes to $x_i$ and with $\zeta_i(\Spec(K_i))\in U$. 
%For each $i$, let $(R_i,\mathfrak{m}_{i})$ be a complete prime regular DVR with $R_i/\mathfrak{m}_{i}=k_i$ and $K_i=\text{Frac}(R_i)$. Then, for any $x_i\in M(k_i)$, there exists $P_i\in M(R_i)$ by Hensel's Lemma. Moreover, since $M\to \Spec(\Z)$ is smooth, there is $\zeta_i: \Spec(R_i)\to M$ which specializes to $x_i$ and with $\zeta_i(\Spec(K_i))\in U$. 
Therefore, $\Xx_M\times_M \Spec(K_i)$ is smooth and separably rationally connected. Since $k_i$ is RC solving, we get that $\Xx_M\times_M \Spec(k_i)$ has a $k_i$-rational point. In particular, $\Xx(k)\neq \emptyset$.\\
%Let $R'$ be prime regular DVR with $R/\mathfrak{m}_{R'}=k'$ and $K'=\text{Frac}(R')$, and $\Xx'$ be a proper, flat $R'$-scheme with $\Xx_{K'}'$ smooth and separably rationally connected. By Fact \ref{primeregparspace}, there exists a parameter space $M$ over $S=\Spec(\Z)$ and a dominant $S$-morphism $\zeta: \Spec(R')\to M$ such that $\Xx'\to \Spec(R')$ is the pullback of $\Xx_M\to M$ via $\zeta$. There is a open subset $U\subseteq M$ such that $\Xx_u$ is smooth and rationally connected. Let $x\in M(k)$. Since $x$ is smooth, there exists $P\in \Xx_M(R)$ such that $P_{\eta}\in U$. We conclude that for each $x\in M(k)$, the $k$-scheme $\Xx_M\times_M x$ has a $k$-rational point. This is encoded by a $\forall \exists$-sentence and therefore also holds in $k'$. 
(ii) By Lemma \ref{easyprops2}(ii), we may replace $(K,v)$ with any separable algebraic extension with the same residue field.
%, in particular a field-theoretic complement of $K^{ur}$ (see \S 5.2 \cite{engprest}). 
We can therefore assume that $(K,v)$ is algebraically maximal with divisible value group. 
%See \cite{engprest}.
Let $(S,\mathfrak{m}_S)$ be an equal characteristic DVR with $S/\mathfrak{m}_S=k$ and $L=\text{Frac}(S)$. Let $\Xx$ be a flat, projective $S$-scheme with $\Xx_{L}$ smooth and separably rationally connected. Arguing as above, there is an extension $K'$ of $L$ which is algebraically maximal with divisible value group and residue field $k$. By Theorem 1.4 \cite{Kuhl}, we get that $K'\equiv K$ in $L_{\text{rings}}$. Since $K$ is geometrically $C_1$, the same is true for $K'$ by Proposition \ref{geomc1elemclass}. 
%By assumption, we get that $L\equiv_{\forall \exists} k$ and hence $L$ is geometrically $C_1$. 
Therefore, $\Xx(K')\neq \emptyset$. By the valuative criterion of properness, we get $\Xx(\Oo_{K'})\neq \emptyset$ and hence $\Xx(k)\neq \emptyset$.
\end{proof}

\bp \label{rcsolvelemclass}
The class of RC solving fields is elementary in $L_{\text{rings}}$. Moreover, it is $\forall \exists$-axiomatizable.
\ep  
\begin{proof}
As in Proposition \ref{geomc1elemclass}, using Lemmas \ref{easylemrc} and \ref{easylemrc2}.
\end{proof}
%The following is due to 

%We note that---prior to Hogadi-Xu---de Jong had proved a version of the above Corollary for fields $k$ containing $\overline{\Q}$ (cf. Theorem 1.4 \cite{starr}). 
%However, we can prove the following which is sufficient for many applications:
\bl \label{geomc1impliesrcstrong}
Let $k$ be a geometrically $C_1$ field such that $k\equiv_{\forall \exists} k(\!(t^{\Gamma})\!)$ in $L_{\text{rings}}$ for some ordered abelian group $\Gamma$. Then $k$ is equal characteristic RC solving. In particular, $k$ is $C_1$.
\el 
\begin{proof}
By Proposition \ref{geomc1elemclass}, the field $k(\!(t^{\Gamma})\!)$ is geometrically $C_1$. By Lemma \ref{easylemrc2}(ii), it follows that $k$ is equal characteristic RC solving.
%Let $R$ be a an equal characteristic DVR with $K=\text{Frac}(R)$ and residue field $k$. Let $\Xx$ be a proper, flat $R$-scheme $\Xx$ with $\Xx_K$ smooth and separably rationally connected. Let $L$ be a maximal totally ramified extension of $K$ and note that $L  \equiv_{\forall \exists} k(\!(t^{\Gamma})\!)$ by Fact \ref{tame}. 
%By assumption, we get that $L\equiv_{\forall \exists} k$ and hence $L$ is geometrically $C_1$. Therefore, $\Xx(L)\neq \emptyset$. By the valuative criterion of properness, we get $\Xx(\Oo_L)\neq \emptyset$. By reduction modulo $\mathfrak{m}_L$, we get that $\Xx(k)\neq \emptyset$.
%$X\subseteq \mathbb{P}^n_k$ be a hypersurface of degree $d\leq n$. Then, there exists a $k[\![t]\!]$-scheme $\Xx$ such that $\Xx_k=X$ and such that the generic fiber $\Xx_{k(\!(t)\!)}$ is a general hypersurface in $\mathbb{P}^n_{k(\!(t)\!)}$ of degree $d\leq n$. 
%Namely, one can lift the coefficients of the defining equation for $X$ such that they are algebraically independent.
%Then $\Xx_{k(\!(t)\!)}$ is separably rationally connected by \cite{zhu}. Since $k\equiv_{\forall \exists} k(\!(t^{\Gamma})\!)$, we get that $k(\!(t^{\Gamma})\!)$ is geometrically $C_1$ by Lemma \ref{geomc1elemclass}. Therefore, there exists $x \in \Xx(k(\!(t^{\Gamma})\!))$ and also $P\in \Xx(k[\![t^{\Gamma}]\!])$. Finally, we get that $P_k\in X(k)$.
\end{proof}

\bc\label{geomc1impliesrc}
Let $(K,v)$ be a nontrivially valued tame field with divisible value group. Suppose that $K$ is geometrically $C_1$. Then $K$ is equal characteristic RC solving. In particular, $K$ is geometrically $C_1$ and also $C_1$.
\ec
\begin{proof}
By model-completeness of the theory of divisible ordered abelian groups, we get that $\Gamma \preceq \Gamma \oplus_{lex} \Q$. By Theorem 1.4 \cite{Kuhl}, it follows that $(K,v) \preceq (K(\!(t^{\Q})\!),v_t\circ v) $. In particular, $K\equiv K(\!(t^{\Q})\!)$ in $L_{\text{rings}}$. Finally, apply Lemma \ref{geomc1impliesrcstrong} to the valued field $(K(\!(t^{\Q})\!),v_t)$ to conclude that $K$ is equal characteristic RC solving. The last statement follows from Lemma \ref{rcimpliesgeomc1}.
\end{proof}
\begin{rem}
It is not known in general if geometrically $C_1$ fields of positive characteristic are $C_1$. Notably, perfect PAC fields of positive characteristic are clearly geometrically $C_1$ but it is not known whether they are also $C_1$.
\end{rem}

\begin{fact} [Theorem 1.1 \cite{esnault2}, Theorem 1.1 \cite{esnaultxu}] \label{finitefieldsRC}
Every finite field is RC solving. 
\end{fact} 
\bc 
Algebraic extensions of finite fields and pseudofinite fields are RC solving. %and pseudofinite fields.
\ec
\begin{proof}
From Fact \ref{finitefieldsRC} and Proposition \ref{rcsolvelemclass}.
\end{proof}
%\begin{fact} [Theorem 1.1 \cite{esnault2}, Theorem 1.1 \cite{esnaultxu}] 
 %Every finite field is RC solving.
%Let $(K,v)$ be a discrete valued field with finite residue field $k$. Let $X$ be a smooth, projective, geometrically irreducible $K$-variety. Suppose that the $\ell$-adic cohomology $H^i(\overline{X})$ is supported in codimension $\geq 1$ for all $i\geq 1$. Let $\Xx$ be an $\Oo_K$-model of $X$. Then $\Xx(k)\neq \emptyset$.
 %Assume that the $\ell$-adic cohomology $H_{\text{\'et}}^i(\overline{X})$ is supported in codimension $i\geq 1$. 
 %Let $\Xx$ be an $R$-model of $X$. Then $\Xx_k(k)\neq \emptyset$.
%\end{fact}

%Let $X$ be a smooth, proper, rationally connected $K$-variety and $\Xx$ be an $\Oo_K$-model of $X$. Then $\Xx(k)\neq \emptyset$.

\begin{fact} [Theorem 3.10 \cite{starr2}]
Function fields of curves over algebraically closed fields are RC solving.
\end{fact}
%It is instructive to review Starr's argument for what follows: 

\subsection{Strongly RC solving fields}
We also consider a stronger variant of \quotes{RC solving} in which we drop the prime regular assumption on the DVR: 
\begin{definition}
A field $k$ is \textit{strongly RC solving} if for every DVR $(R,\mathfrak{m}_R)$ with $R/\mathfrak{m}_R\subseteq k$ and $K=\text{Frac}(R)$ and for every flat, projective $R$-scheme $\Xx$ with $\Xx_K$ smooth and separably rationally connected, we have that $\Xx(k)\neq \emptyset$. 
\end{definition}
\begin{rem} \label{strongrcrem}
\begin{enumerate}[label=(\roman*)] 

\item By Lemma 3.4 \cite{starr2}, any DVR of equal characteristic is prime regular. Therefore, the above definition only makes a difference (if any) in the mixed characteristic case. 
\item Note that a DVR with finite residue field is automatically prime regular, hence algebraic extensions of finite fields are automatically strongly RC solving.

\item We do not know if function fields of curves over algebraically closed fields are strongly RC solving.

\item We do not know if strongly RC solving fields form an elementary class. 
\end{enumerate}

\end{rem}

%We also  do not know if RC solving fields are strongly RC solving.
\section{Adic spaces} \label{adicspaces}
%We recall some facts about adic spaces. 
Let $(K,v)$ be a complete valued field of rank $1$ and $X$ be a $K$-variety. One can associate to $X$ a space of valuations $X^{\ad}$, called its adic space. This space admits an alternative description as the projective limit of all formal $\Oo_K$-models of $X$. These two descriptions will allow us to shift from $\Oo_K$-models to valuations. When $(K,v)$ is tame with divisible value group, we deduce from Kuhlmann's theory that $X(K)\neq \emptyset$ provided that $\Xx(k)\neq \emptyset$ for every $\Oo_K$-model $\Xx$.
%, using .  
%\S 2 \cite{Scholze} 
\subsection{Integral models}
We briefly recall some background but refer the reader to \S 3 \cite{conradseveral} and \S 8 \cite{boschbook} for details.
Let $(K,v)$ be a complete valued field of rank $1$ (not necessarily discrete) with valuation ring $\Oo_K$ and residue field $k$. 
\begin{definition}
Let $X$ be a proper, separated $K$-scheme of finite type.
\begin{enumerate}[label=(\roman*)]
\item  An $\Oo_K$-model $\Xx$ of $X$ is a flat, proper, separated $\Oo_K$-scheme of finite presentation together with an isomorphism  of $K$-schemes $\iota: \Xx_K\to X$. 
\item A morphism of $\Oo_K$-models $f:\Xx\to \Xx'$ is a morphism of $\Oo_K$-schemes such that $f_K:\Xx_K \to \Xx_K'$ is an isomorphism compatible with the isomorphisms $\iota:\Xx_K\to X$ and $\iota':\Xx_K'\to X$. 
\end{enumerate}
\end{definition}
Let $X^{\rig}$ be the rigid-analytic space over $K$ associated to $X$. Then, we have the notion of a formal $\Oo_K$-model of $X^{\rig}$. More generally:
\begin{definition}
Let $X$ be a rigid-analytic space over $K$.
\begin{enumerate}[label=(\roman*)]
\item A formal $\Oo_K$-model of $X$ is an admissible formal $\Oo_K$-scheme $\mathfrak{X}$ together with a $K$-isomorphism $\iota: \mathfrak{X}_K\to X$. 
\item A morphism of formal $\Oo_K$-models $f:\mathfrak{X}\to \mathfrak{X}'$ is a morphism of formal $\Oo_K$-schemes such that $f_K:\mathfrak{X}_K \to \mathfrak{X}_K'$ is an isomorphism compatible with the isomorphisms $\iota: \mathfrak{X}_K\to X$ and $\iota':\mathfrak{X}_K'\to X$. 
\end{enumerate}
\end{definition}
We write $\mathcal{M}_X$ for the category of admissible formal $\Oo_K$-models of $X^{\rig}$. The category $\mathcal{M}_X$ of admissible formal $\Oo_K$-models of $X^{\rig}$ is filtered: 
\bl [\S 8.4, Lemma 4 \cite{boschbook}]
Let $X$ be a proper, separated $K$-scheme of finite type.
\begin{enumerate}[label=(\roman*)]
\item Given two formal $\Oo_K$-models $\mathfrak{X}$ and $\mathfrak{X'}$, there exists at most one morphism of formal $\Oo_K$-models $f:\mathfrak{X}\to \mathfrak{X}'$ .

\item Given two formal $\Oo_K$-models $\mathfrak{X}$ and $\mathfrak{X}'$, there exists a formal $\Oo_K$-model $\mathfrak{X}''$ and morphisms of formal $\Oo_K$-models $\mathfrak{X}''\to \mathfrak{X}$ and $\mathfrak{X}''\to \mathfrak{X}'$. 
\end{enumerate}
\el 
Given an $\Oo_K$-model $\Xx$ of $X$, we naturally obtain a formal $\Oo_K$-model of $X^{\rig}$, namely its \textit{formal completion} $\widehat{\Xx}= \varinjlim \Xx_n$, where $\Xx_n=\Xx\times_{\Oo_K} \Oo_K/\varpi^n$. 
\begin{definition}
A formal $\Oo_K$-scheme $\mathfrak{X}$ is \textit{algebraizable} if there is a finitely presented $\Oo_K$-scheme $\Xx$ and an isomorphism of formal $\Oo_K$-schemes $f: \widehat{\Xx} \to \mathfrak{X}$. 
\end{definition}
We note that the algebraizable formal $\Oo_K$-models are cofinal in $\mathcal{M}_X$:
\begin{fact} \label{cofinal}
For each formal $\Oo_K$-model $\mathfrak{X}$ of $X^{\rig}$, there exists an $\Oo_K$-model $\Xx$ of $X$ and a morphism of formal $\Oo_K$-schemes $\widehat{\Xx}\to \mathfrak{X}$.
%In particular, we have a canonical isomorphism
%$$\varprojlim_{\Xx\in \mathcal{M}_X}  \Xx_k \cong \varprojlim_{\mathfrak{X}_i \in \mathcal{M}_X} \mathfrak{X}_i$$
\end{fact}
\begin{proof}
By \S 8.4, Proposition 6 \cite{boschbook}, the formal blowups are algebraizable and are cofinal in $\mathcal{M}_X$ by \S 8.4, Lemma 4 \cite{boschbook}.
\end{proof}

Recall that $X^{\ad}$ comes equipped with two presheaves $\Oo_{X^{\ad}}$ and $\Oo_{X^{\ad}}^+$ (see \S 2 \cite{Scholze}). Given $x\in X^{\ad}$, we write $\kappa(x)$ for the residue field of $\Oo_{X^{\ad},x}$ and $\kappa(x)^+\subseteq \kappa(x)$ for the image of $\Oo_{X,x}^+$ in $\kappa(x)$, which is a valuation ring of $\kappa(x)$ extending $\Oo_K$. 
%The adic space $X^{\ad}$ associated to $X$ is isomorphic to the inverse limit of $\mathcal{M}_X$:
\begin{fact} [Raynaud]\label{adicinverselimit}
%Let $X$ be a qcqs adic space locally of finite type over $K$. By Raynaud, there exist formal $\Oo_K$-models $\mathfrak{X}$ for $X$, unique up to admissible blowup. 
Let $X$ be a separated $K$-scheme of finite type.
Then, there is a homeomorphism 
$$X^{\ad}\cong \varprojlim_{\mathfrak{X} \in \mathcal{M}_X} |\mathfrak{X}_i |$$
extending to an isomorphism of locally ringed topological spaces 
$$(X^{\ad},\Oo_{X^{\ad}}^+)\cong \varprojlim_{\mathfrak{X}_i} (\mathfrak{X}_i,\Oo_{\mathfrak{X}_i})$$ 
where the right hand side is the inverse limit in the category of locally ringed spaces.
\end{fact}
\begin{proof}
See Theorem 3 \cite{boschbook} or Theorem 2.22 \cite{Scholze}.
\end{proof}
\subsection{Compactness and inverse limits}
Let $(X_i,f_{ij})_{i\in I}$ be an inverse system of finite type $k$-schemes and $X=\varprojlim X_i$. We will typically be in a situation where we know that $X_i(k)\neq \emptyset$ for each $i$ but we would really like to have that $X(k)\neq \emptyset$. A standard compactness argument shows that this is possible, except that we may need to replace $k$ by an elementary extension. More generally:
%have a point in each of these schemes.  
\bl \label{gensaturatedzr}
Let $R$ be a commutative ring, $(I,\leq)$ be a directed set, $(X_i,f_{ij})_{i\in I}$ be an inverse system of finitely presented $R$-schemes and $X=\varprojlim X_i$. Let $A$ be an $R$-algebra such that $X_i(A)\neq \emptyset$ for each $i\in I$. Then, for every $(|I|+|R|)^+$-saturated elementary extension $A\preceq A^*$, we have that $X (A^*)\neq \emptyset$.
\el 
\begin{proof}
For simplicity, assume that each $X_i=\Spec(B_i)$ is affine, where $B_i$ is a finitely presented $R$-algebra, say with $n_i$ generators. The general case only involves some additional bookkeeping.
%is left to the reader. 
For each $i\in I$, we introduce a formal $n_i$-tuple of variables $x_i=(x_{i,1},...,x_{i,n_i})$. Consider the set of formulas  $p(x_{i}: i\in I )$ in $L_{\text{rings}}(R)$ in the variables $x_i$, saying that $x_i \in X_i(A)$ and that $f_{ij}(x_i)=x_j$. Given a finite subset $S\subseteq p(x_{i}: i\in I )$, there is a finite subset $I_0 \subseteq I$ such that $S$ refers only to variables $x_i$ with $i\in I_0$. Since $(X_i,f_{ij})$ is an inverse system. there is $i_0 \in I$ such that $X_{i_0}$ dominates all $X_i$ with $i\in I_0$. By assumption, there is $a_0\in X_{i_0}(A)$. This gives compatible $k$-rational points $a_i\in X_i(k)$ for $i\in I_0$, namely $a_i=f_{i_0i}(a_0)$. It follows that $p(x_{i}: i\in I )$ is a partial type and can be realized in $A^*$ by $(|I|+|R|)^+$-saturation. We conclude that $X(A^*)\neq \emptyset$. 
\end{proof}

\begin{rem}
We sketch a different proof for the case where $A$ is a finite ring and $(I,\leq )=(\N,\leq)$, which is sufficient for our results in \S \ref{maxtotloc}. In that case, the conclusion is that $X (A)\neq \emptyset$ (a finite structure has no proper elementary extensions). We prove this using K\"onig's Lemma: Consider the graph $G$ whose set of vertices is given by
$$V(G) =\bigsqcup_{i\in \N} X_i(A) $$
We draw an edge from $a_i\in X_i(A)$ to $a_j\in X_j(A)$ if $i\geq j$ and $f_{ij}(a_i)=a_j$. Consider $G\cup \{*\}$, where $*$ is a dummy element which connects precisely to the elements of $X_0(A)$. One sees that $G\cup \{*\}$ is a connected, locally finite, infinite graph and therefore has an infinite ray. Such a ray corresponds to a sequence of points $a_i\in X_i(A)$ such that $f_{ij}(a_i)=a_j$ for $i\geq j$, i.e., to an element $(a_i) \in X(A)$.
%$a\in X(A)$.
\end{rem}

\bp \label{relativezr1}
Let $(K,v)$ be a complete valued field of rank $1$ with residue field $k$. Let $X$ be a $K$-variety and suppose that $\Xx(k)\neq \emptyset$ for each $\Oo_K$-model $\Xx$ of $X$. Then there exists $x \in X^{\ad}$ such that $k$ is existentially closed in the residue field of $\kappa(x)^+$.
\ep 
\begin{proof}
To apply Lemma \ref{gensaturatedzr}, we first need to restrict to an inverse system which is cofinal in $\mathcal{M}_X$ and is indexed by a set (not a proper class). For instance, consider the admissible blowups of a fixed formal model. This collection is indexed by a set of size $|K|$. 

Let $k^*$ be a $|K|^+$-saturated elementary extension of $k$. By Lemma \ref{gensaturatedzr}, we have that $\varprojlim \Xx_i(k^*)\neq \emptyset$, where the inverse limit is taken over all $\Oo_K$-models of $X$. By Fact \ref{cofinal}, each formal $\Oo_K$-model $\mathfrak{X}$ is dominated by an algebraizable one, hence  $\varprojlim \mathfrak{X}_i(k^*)\neq \emptyset$.  
%that $\varprojlim_{\Xx\in \mathcal{M}_X}  \Xx(k^*)\neq \emptyset$. By Fact \ref{cofinal}, we also get 
This gives rise to a sequence of points
$$(x_i)\in \varprojlim_{\mathfrak{X}_i \in \mathcal{M}_X} \mathfrak{X}_i$$ 
By Fact \ref{adicinverselimit}, the sequence $(x_i)$ corresponds to some
$x \in X^{\ad}$ such that
$$ \Oo_{X^{\ad},x }^+ \cong \varinjlim \Oo_{\mathfrak{X}_i,x_i} $$
%Let $\kappa(x)$ be the residue field of $\Oo_{X^{\ad},x }$ and $\kappa(x)^+$ 
%/\mathfrak{m}_{X^{\ad},\xi}$ be the residue field of $X^{\ad}$ at $x$ .
% and $\kappa(\xi)^+$ be the image of $\Oo_{X^{\ad},\xi}^+$ in $\kappa(\xi)$. 
%Then $(\kappa(x),v_x)$ is a valued field extending $(K,v)$, with valuation
In particular, the residue field $k'$ of $\kappa(x)^+$ is equal to $\varinjlim \kappa(x_i)$. Now each $\kappa(x_i)$ embeds into $k^*$ and these embeddings are compatible with the maps of the directed system of the $\kappa(x_i)$'s. This allows us to identify $k'$ with a subfield of $k^*$. Since $k\subseteq k' \subseteq k^*$ and $k\preceq k^*$, it follows that $k\preceq_{\exists} k'$.
%This gives a semi-valuation on $\Oo_{X,x}$. Let $U=\Spec(R)$ be an affine open containing $x$. We have a semi-valuation on $R$, hence a valuation on $\kappa(x)=\text{Frac}(R/\mathfrak{p})$ where $\mathfrak{p}=\text{supp}(v_{\xi})$. The residue field is a subfield $k'$ of $\lim \kappa(x_i)$.
\end{proof}

\bc \label{relativezr}
Let $(K,v)$ be a valued field of rank $1$ with residue field $k$. Let $X$ be a $K$-variety and suppose that $\Xx(k)\neq \emptyset$ for each $\Oo_K$-model $\Xx$ of $X$. Then there exists $\xi \in X$ and a valuation $v_{\xi}$ on $\kappa(\xi)$ whose residue field $k'$ is such that $k\preceq_{\exists} k'$.
\ec
\begin{proof}
It is harmless to replace $(K,v)$ with its completion, hence we assume that $(K,v)$ is complete. Let $x \in X^{\ad}$ be as in Proposition \ref{relativezr1}. We have a natural analytification map of locally ringed spaces 
$$i: (X^{\ad}, \Oo_{X^{\ad}}) \to (X,\Oo_X) $$
Let $\xi=i(x)$ and $\kappa(\xi)\subseteq \kappa(x)$ be the induced inclusion. We let $v_{\xi}$ be the restriction of $v_{x}$ to $\kappa(\xi)$ and note that it satisfies the desired properties.
\end{proof}

\subsection{Rational points over tame fields}
%Finally, we specialize to the case of tame fields with divisible value group:
We now apply our previous results to the case of tame fields. We refer the reader to the original article by F.-V. Kuhlmann \cite{Kuhl} for relevant definitions and facts.
%algebra and model theory of tame fields was introduced and studied by  . 
\begin{definition}
A valued field $(K,v)$ is \textit{tame} if it is algebraically maximal, with perfect residue field and $p$-divisible value group, where $p$ is the characteristic exponent of the residue field. 
\end{definition}
The following is an Ax-Kochen/Ershov principle for existential closedness:
%We will need some model-theoretic results due to Kuhlmann \cite{Kuhl}:
%This generalizes the reduction from general henselian valued fields with divisible value group to Puiseux series that was stated in the introduction.
%If $\Oo_K$ is a direct limit of prime regular DVRs, then clearly $(K,v)$ is prime regular. The converse is not true in general, e.g., if $\Gamma_K$ is of rank $>1$. Nevertheless, for tame fields with divisible value group we can prove a kind of converse up to elementary equivalence. The key ingredient is the following:
%Given a valued field $(K,v)$, we write $(K^h,v^h)$ for the henselization, which is unique up to (non-unique) isomorphism.
% The following existential version will be crucial for Theorem \ref{tamemodels}:

\begin{fact}[Theorem 1.4 \cite{Kuhl}] \label{existentialake}
Let $(K,v)$ be a tame valued field and let $(K',v')/(K,v)$ be a valued field extension. Then, the following are equivalent: 
\begin{enumerate}[label=(\roman*)]
\item $(K,v)\preceq_{\exists} (K',v')$ in $L_{\text{val}}$. 

\item \label{condition2}  $k\preceq_{\exists} k'$ in $L_{\text{rings}}$ and $\Gamma\preceq_{\exists} \Gamma'$ in $L_{\text{oag}}$.
\end{enumerate}
\end{fact}

\begin{rem}
%
%\item Note $(K',v')$ is also assumed to be tame in Theorem 1.4 \textit{ibid}. A small argument shows that this is not needed: Any $(K',v')$ can be embedded in a tame field $(K'',v'')$ such that $k''$ is the perfect hull of $k'$ and $\Gamma''$ is the $p$-divisible hull of $\Gamma'$, where $p$ is the characteristic exponent of $k$. 
%For instance, take K'' to be a maximal immediate extension of a maxima purely wild extension of K.
%Since $k$ is perfect and $\Gamma$ is $p$-divisible, we still have that $k\preceq_{\exists} k''$ and $\Gamma\preceq_{\exists} \Gamma''$. By Theorem 1.4 \textit{ibid}, we get that $(K,v)\preceq_{\exists} (K'',v'')$ in $L_{\text{val}}$. Since $(K',v')\subseteq (K'',v'')$, we get that $(K,v)\preceq_{\exists} (K',v')$ in $L_{\text{val}}$. 
%Proposition 32 in \cite{kuhlplaces}.
%\item 
%This has the effect of making the value group $p$-divisible and 
 If in addition $\Gamma$ is divisible and nontrivial, then the condition $\Gamma\preceq_{\exists} \Gamma'$ holds automatically since nontrivial divisible ordered abelian groups are existentially closed.
%\end{rem}
%\end{enumerate}

\end{rem}

%In the case of tame fields with divisible value group, one even gets a $K$-point on $X$: 
\bt \label{tamemodels}
Let $(K,v)$ be a tame valued field with divisible value group of rank $1$ and $X$ be a proper $K$-variety. Then, the following are equivalent: 

\begin{enumerate}[label=(\roman*)] 
\item $X(K)\neq \emptyset$.

\item $\Xx(k)\neq \emptyset$, for each $\Oo_K$-model $\Xx$ of $X$. 
\end{enumerate}

\et  
\begin{proof}
(i)$\Rightarrow$(ii): This follows directly from the valuative criterion of properness and does not require any special assumptions on $(K,v)$.\\
%Let $\Xx$ be an $\Oo_K$-model of $X$. 
%By the valuative criterion of properness, we have that $\Xx(\Oo_K)\neq \emptyset$. By reduction modulo $\mathfrak{m}$, we get that $\Xx(k)\neq \emptyset$.\\
(ii)$\Rightarrow$(i): By Corollary \ref{relativezr}, there exists a scheme-theoretic point $\xi \in X$ and a valuation $v_{\xi}$ on $\kappa(\xi)$ whose residue field $k'$ is such that $k\preceq_{\exists} k'$. By Fact \ref{existentialake}, we get that $(K,v)\preceq_{\exists} (\kappa(\xi),v_{\xi})$ in $L_{\text{val}}$. In particular, we get that $K\preceq_{\exists} \kappa(\xi)$ in $L_{\text{rings}}$. Equivalently, the Zariski closure $\overline{\{\xi\}}\subseteq X$ has Zariski dense many $K$-rational points. We conclude that $X(K)\neq \emptyset$.
\end{proof}
%\bl 
%Let $k$ be a field of characteristic zero, $R=k[\![t]\!]$ and $K=\text{Frac}(R)$. Let $X$ be a smooth projective rationally conenected variety and $\Xx$ be a strict normal crossings $R$-model of $X$. Then, every irreducible component of $\Xx_k$ of multiplicity $1$ is rationally connected.
%\el  
%\begin{proof}
%Since $Y$ is smooth and $k$ is of characteristic $0$, it suffices to show that $Y$ is rationally chain connected (see Theorem 3.10 \cite{kollarbook}). By 
%We can assume that $k$ is algebraically closed. Write $f:\Xx\to R$ for the structure morphism. Let $Y$ be such an irreducible component. Since $Y$ is smooth, it suffices to show that $Y$ is rationally chain connected. Let $U\subseteq Y$ be open subset of points $Y$  which do not lie in any other irreducible component of $\Xx_k$. Then $f$ is smooth at any $x\in U$. Given $x,x'\in U$, we can lift them to $R$-points $P,P'$ of $\Xx$. Since $\Xx_K$ is rationally connected, there is $f: \mathbb{P}^1_K \to \Xx_K$ connecting $P_K$ and $P'_K$. This rational curve can break up into a chain of rational curves $C_1\cup C_2...\cup C_n$ under specialization. Each $C_i$ is contained in a single irreducible component of $\Xx_k$ and by induction, we see that each $C_i$ is contained in $Y$. 
%\end{proof}
\section{Transfer theorems }
We prove a transfer principle for \quotes{geometrically $C_1$} in equal characteristic $0$. We also prove that tame fields with divisible value group and finite residue field are geometrically $C_1$. Both of these statements are special cases of a general transfer principle in arbitrary characteristic.
%We first prove a general transfer principle in arbitrary characteristic. 
%\subsection{Rank $1$ models}
\subsection{A general transfer principle}
%We recall some basic facts from tame fields:
\begin{fact}[Theorem 1.4 \cite{Kuhl}] \label{existentialake}
Let $(K',v')/(K,v)$ be an extension of tame fields. Then, the following are equivalent: 
\begin{enumerate}[label=(\roman*)]
\item $(K,v)\preceq (K',v')$ in $L_{\text{val}}$. 

\item \label{condition2}  $k\preceq  k'$ in $L_{\text{rings}}$ and $\Gamma\preceq  \Gamma'$ in $L_{\text{oag}}$.
\end{enumerate}
\end{fact} 
\begin{rem}
If in addition $\Gamma$ and $\Gamma'$ are divisible, then  $\Gamma\preceq  \Gamma'$ holds automatically by model-completeness of the theory of divisible ordered abelian groups.
\end{rem}

\bl \label{unionofdvr}
Let $(K,v)$ be a nontrivially valued tame field. Then, there exists a valued subfield $(K',v')\subseteq (K,v)$ such that: 
\begin{enumerate}[label=(\roman*)]
\item $(K',v')$ is a tame field.
\item $\Gamma/\Gamma'$ is torsion-free and $k=k'$.
%$(K',v')\preceq (K,v)$ in $L_{\text{val}}$.
\item $\Oo_{K'}$ is a direct limit of DVRs. 
\end{enumerate}
If in addition $\Gamma$ is divisible, then $(K',v')\preceq (K,v)$ in $L_{\text{val}}$.
\el 
\begin{proof}
We first construct a discrete valued subfield $(K_1,v_1)\subseteq (K,v)$ with residue field $k_1$ such that $k/k_1$ is separable algebraic. Consider the prime subfield $K_0$ of $K$, equipped with the restriction $v_0$ of $v$.  Choose a separating transcendence basis $\tau=\{t_i:i\in I\}$ for $k/k_0$. Let $T_i$ be a lift of $t_i$ in $K$ and $\mathrm{T}=\{T_i:i\in I\}$. We now let $K_1=K_0(\mathrm{T})$ in case $(K,v)$ is of mixed characteristic, and $K_1=K_0(\mathrm{T} \cup \{t\})$ in case $(K,v)$ is of equal characteristic, where $t$ is any element in the maximal ideal of $K$. In each case, we endow $K_1$ with the restriction $v_1$ of $v$.
%we let $K_1=K_0(\mathrm{T})$ be equipped with the restriction $v_1$ of $v$.  If $(K,v)$ is of equal characteristic, we let $K_1=K_0(\mathrm{T} \cup \{t\})$, where $t$ is any element in the maximal ideal of $K$.
By Lemma 2.2 \cite{Kuhl}, we get that $(K_1,v_1)$ is a discrete valued field with residue field $k_0(\tau)$. 
%can be obtained by choosing a separating transcendence basis $\tau$ for $k$ over its prime subfield and 

Let $K'\subseteq K$ be the relative algebraic closure of $K_1$ in $K$. We endow $K'$ with the restriction $v'$ of $v$. By Lemma 3.7 \cite{Kuhl}, we get that $(K',v')$ satisfies (i) and (ii). It is clear that $\Oo_{K'}$ is a direct limit of DVRs. %Note that $\Gamma_L\preceq \Gamma_K$ by model-completeness of the theory of divisible ordered abelian groups. Write $(L,v)=\varinjlim (L_{i},v_{i})$, where $\Oo_{L_{i}}$ is a prime finite DVR for each $i\in I$. Let $(K',v')$ be as in Lemma \ref{elimramlemm}, so $(K',v')\equiv (K,v)$ in $L_{\text{val}}$ and $k'=k$.
If in addition $\Gamma$ is divisible, then $\Gamma'$ must also be divisible by (ii).
%Since $\Gamma/\Gamma'$ is torsion-free, we get that $\Gamma'$ is divisible. 
We then have that $\Gamma' \preceq \Gamma$ by model-completeness of DOAG. %the theory of divisible ordered abelian groups. 
By Fact \ref{existentialake}, we conclude that $(K',v') \preceq (K,v)$  in $L_{\text{val}}$. 
\end{proof}

\bt \label{mainthm1}
Let $(K,v)$ be a tame valued field with divisible value group and residue field $k$.  

\begin{enumerate}[label=(\roman*)]
\item Assume that $k$ is strongly RC solving. Then $K$ is geometrically $C_1$ and also $C_1$. 

\item Assume that $\text{char}(K)=\text{char}(k)$ and $k$ is RC solving. Then $K$ is geometrically $C_1$ and also $C_1$. 

\end{enumerate}
\et 
\begin{proof}
(i) If $v$ is the trivial valuation, the conclusion follows from Lemma \ref{rcimpliesgeomc1}. We can therefore assume that $v$ is non-trivial. By Proposition \ref{unionofdvr}, there exists $(K',v')\preceq (K,v)$ with $k'=k$ such that $\Oo_{K'}$ is a direct limit of DVRs. By Proposition \ref{geomc1elemclass}, it suffices to show that $K'$ is geometrically $C_1$. Therefore, upon replacing $K$ with $K'$, we can assume that $\Oo_K$ is a direct limit of DVRs, say $\Oo_K=\varinjlim_{i\in I} \Oo_{K_i}$. Let $X$ be a smooth projective separably rationally connected $K$-variety and $\Xx$ be an $\Oo_K$-model of $X$.  Then there exists $i \in I$ and a flat, proper $\Oo_{K_i}$-scheme $\Xx_i$ such that $\Xx=\Xx_i \times_{\Oo_{K_i}} \Oo_K$. Since $k$ is strongly RC solving, we get that $\Xx(k)\neq \emptyset$. 
%(i) Let $X$ be a smooth projective separably rationally connected $K$-variety and $\Xx$ be an $\Oo_K$-model of $X$.  Then there exists $i \in I$ and a flat, proper $\Oo_{K_i}$-scheme $\Xx_i$ such that $\Xx=\Xx_i \times_{\Oo_{K_i}} \Oo_K$. By Corollary \ref{primeregularafterfinext}, we get that $\Xx(k)\neq \emptyset$. \\
%Since $\Oo_{K_i}$ is a prime regular DVR and $k$ is RC solving, we get that $\Xx_i(k)\neq \emptyset$, hence $\Xx(k)\neq \emptyset$. 
By Theorem \ref{tamemodels}, we get that $X(K)\neq \emptyset$. It follows that $K$ is geometrically $C_1$. By Lemma \ref{geomc1impliesrc}, we get that $K$ is $C_1$.\\
(ii) By Lemma 3.4 \cite{starr2}, any DVR of equal characteristic is prime regular. One then proceeds as in part (i).
%(ii) Now let $X$ be a smooth projective separably rationally connected $K$-variety. For any $\Oo_K$-model $\Xx$ of $X$, there exists $i \in I$ and a flat, proper $\Oo_{K_i}$-scheme $\Xx_i$ such that $\Xx=\Xx_i \times_{\Oo_{K_i}} \Oo_K$. Since $\Oo_{K_i}$ is a prime regular DVR and $k$ is RC solving, we get that $\Xx_i(k)\neq \emptyset$, hence $\Xx(k)\neq \emptyset$. By Theorem \ref{tamemodels}, we get that $X(K)\neq \emptyset$. It follows that $K$ is geometrically $C_1$. By Lemma \ref{geomc1impliesrc}, we get that $K$ is $C_1$.
\end{proof}

\subsection{Equal characteristic $0$} \label{equal0sec}
\bt  
 Let $(K,v)$ be a henselian valued field of equal characteristic $0$ with divisible value group. Then:
 $$K\mbox{ is geometrically }C_1\iff k \mbox{ is geometrically }C_1 $$
 In particular, if $k$ is geometrically $C_1$, then $K$ is $C_1$.
\et  
\begin{proof}[First proof]
If $K$ is geometrically $C_1$, then $k$ is geometrically $C_1$ by Lemma \ref{easyprops2}(iii). The converse follows from Lemma \ref{hogadi} and Theorem \ref{mainthm1}. 
\end{proof}
%Let us present an alternative proof for the converse:
%\bl 
%Let $k$ be a field of characteristic zero, $R=k[\![t]\!]$ and $K=\text{Frac}(R)$. Let $X$ be a smooth projective rationally conenected variety and $\Xx$ be a strict normal crossings $R$-model of $X$. Then, every irreducible component of $\Xx_k$ of multiplicity $1$ is rationally connected.
%\el  
%\begin{proof}
%Since $Y$ is smooth and $k$ is of characteristic $0$, it suffices to show that $Y$ is rationally chain connected (see Theorem 3.10 \cite{kollarbook}). By 
%We can assume that $k$ is algebraically closed. Write $f:\Xx\to R$ for the structure morphism. Let $Y$ be such an irreducible component. Since $Y$ is smooth, it suffices to show that $Y$ is rationally chain connected. Let $U\subseteq Y$ be open subset of points $Y$  which do not lie in any other irreducible component of $\Xx_k$. Then $f$ is smooth at any $x\in U$. Given $x,x'\in U$, we can lift them to $R$-points $P,P'$ of $\Xx$. Since $\Xx_K$ is rationally connected, there is $f: \mathbb{P}^1_K \to \Xx_K$ connecting $P_K$ and $P'_K$. This rational curve can break up into a chain of rational curves $C_1\cup C_2...\cup C_n$ under specialization. Each $C_i$ is contained in a single irreducible component of $\Xx_k$ and by induction, we see that each $C_i$ is contained in $Y$. 
%\end{proof}
\begin{proof}[Second proof]
The conclusion is clear when $v$ is the trivial valuation, so assume that $v$ is non-trivial. By Ax-Kochen/Ershov in equal characteristic $0$, we get that $K\equiv \bigcup_{n\in \N} k(\!(t^{1/n})\!)$ in $L_{\text{rings}}$. By Proposition \ref{geomc1elemclass}, being geometrically $C_1$ is an elementary property and we can therefore assume that  $K= \bigcup_{n\in \N} k(\!(t^{1/n})\!)$. 
The rest of the argument follows closely the proof of Lemme 7.5 \cite{witt}. 

Let $X$ be a smooth projective rationally connected $K$-variety. For $n\in \N$, write $R_n=k[\![t^{1/n}]\!]$ and $K_n=k(\!(t^{1/n})\!)$. Without loss of generality, we can assume that $X$ is defined over $K_1$. By semistable reduction in equal characteristic $0$ (see Lemme 7.5 \cite{witt}), there is $n\in \N$ such that $X\times_K K_n$ admits a strict semistable $R_n$-model $\Xx$. 
%Using the result of Hogadi-Xu \cite{hogadi} and arguing as in , the special fiber $\Xx_k$ contains a rationally connected subvariety $Z$. 
%By Remark 5.2.4 \cite{brownrational}, there is an irreducible component $Y$ of the special fiber $\Xx_k$ which is rationally connected. 
%Since $k$ is geometrically $C_1$, there exists $x\in Y(k)$. 
By Lemma \ref{hogadi}, the field $k$ is RC solving and hence $\Xx_k(k) \neq \emptyset$. Alternatively, by Remark 5.2.4 \cite{brownrational}, there is an  irreducible component of $\Xx_k$ which is rationally connected and hence has a $k$-rational point. 

Let $x\in \Xx_k(k)$ and $Y_1,...,Y_k$ be the irreducible components of $\Xx_k$ passing through $x$. By the \'etale local description of strict semistable models (see Proposition 2.1.5 \cite{nicaise}), there is a Zariski neighborhood $U\subseteq \Xx$ containing $x$ such that the structure morphism $U\to \Spec(R)$ factors through an \'etale map of the form
\[
  \begin{tikzcd}
    U \arrow[r, "\text{\'et}"] & \Yy=\Spec(R_n[T_1,...,T_m]/(T_1\cdot...\cdot T_k-t^{1/n}) ) 
     \end{tikzcd}
\]
for some $k,n \in \N$ with $k\leq n$. Moreover, this map sends $x$ to the origin of $\mathbb{A}^m_{R_n}$. Note that $\Yy$ has an $R_{kn}$-integral point specializing to the origin, e.g., $$T_1=T_2=...=T_m=t^{1/kn}$$ 
Since $U\to \Yy$ is \'etale at $x$, this lifts to an $R_{kn}$-point of $U$ specializing to $x$ by Hensel's Lemma. We then get that $X(K_{km})\neq \emptyset$. In particular, $X(K)\neq \emptyset$.
%Arguing as in the proof of Lemme 4.6 \cite{witt}, we may assume that $x$ does not lie in any of the other irreducible components $Y_j$ of $\Xx_k$, namely by replacing $\Xx$ with its blowup at $x$ and replacing the point $x$ with an appropriate preimage in the exceptional divisor. Therefore, the morphism $\Xx\to \Spec(R_m)$ is smooth at $x$. By Hensel's Lemma, there exists $P\in \Xx(R_m)$ specializing to $x$. We then get that $X(K_m)\neq \emptyset$ and in particular $X(K)\neq \emptyset$.
%Let $f:\Xx' \to \Xx$ be the blowup of $\Xx$ at $x$. Since $\Xx$ is regular at $x$, the fiber $f^{-1}(\{x\})$ is a projective space and there exists $x'\in f^{-1}(\{x\})$ with residue field $k$, which does not lie in any other irreducible component of $\Xx'_k$.   
%and Lemma \ref{equalcharegular}. 
\end{proof}
%The divisibility of the value group is essential:
%\begin{rem}
%If $\Gamma$ is not divisible and $k$ is not algebraically closed, then $K$ is not $C_1$. To see this, let $\ell$ be such that $p \nmid vt$ for some $t\in K$. We can assume that $G_k$ is an $\ell$-group. This means that $k$ has an extension of degree $\ell^n$ for each $n\in \N$. Write $p$ in $\ell$-basis. 
%\end{rem}
%We propose the following problem which is also 
This raises the following question, which is also implicit in \cite{witt}:
\begin{problem}
Let $k$ be a $C_1$ field of characteristic $0$ and $K= \bigcup_{n\in \N} k(\!(t^{1/n})\!)$ be the Puiseux series field over $k$. Is $K$ also $C_1$?
\end{problem}
\subsection{Finite residue field} \label{maxtotloc}
\bt  \label{hendefdiv}
Let $(K,v)$ be an algebraically maximal valued field with divisible value group and residue field $k$ which is algebraic over a finite field. Then, $K$ is geometrically $C_1$ and also $C_1$.
\et   
\begin{proof}
By Fact \ref{finitefieldsRC}, we have that $k$ is strongly RC solving. 
%(see also Remark \ref{strongrcrem}). 
We conclude from Theorem \ref{mainthm1}.
%By Lemma \ref{algebraicapprox}, there exists $(K',v')\preceq (K,v)$ such that $K'$ is algebraic over a discrete valued field whose residue field is finite. We conclude from Proposition \ref{geomc1elemclass} and \ref{algebraicc1}.
\end{proof}

\bc 
Let $\Gamma$ be a divisible value group and $\F_q$ be a finite field. Then the Hahn series field $\F_q(\!(t^{\Gamma})\!)$ is geometrically $C_1$ and also $C_1$. 
\ec 
\begin{proof}
Recall that Hahn series fields are maximal  (see Corollary 4.13 \cite{vdd}). We conclude from Theorem \ref{hendefdiv}.
\end{proof}

\bc
Any maximal totally ramified extension $K$ of a non-archimedean local field is geometrically $C_1$ and also $C_1$.
\ec
\begin{proof}
Let $K$ be a maximal totally ramified extension of a non-archimedean local field $L$. Recall that $K$ is a field-theoretic complement of $L^{ur}$ and hence every finite extension of $K$ comes from a separable extension of the residue field. Equivalently, by Ostrowski's Lemma, $K$ is henselian defectless with divisible value group. We conclude from Theorem \ref{hendefdiv}.
\end{proof}
%\begin{rem}
%Note that this even holds for smooth projective varieties such that $H^i(X)=0$.
%\end{rem}
 
\begin{rem}
\begin{enumerate}[label=(\roman*)]
\item 
Let $K$ be a non-maximal totally ramified extension of a non-archimedean local field. We claim that $\text{cd}(G_K)\geq 2$, so in particular $K$ is not $C_1$ (see Corollary, pg. 80 \cite{serregal}). Indeed, 
%let $K$ be a maximal totally ramified extension of $L$  and 
consider the short exact sequence
$$1\longrightarrow G_{K^{ur}} \longrightarrow G_K \longrightarrow G_k \longrightarrow 1 $$
%We can assume that $G_K$ is an l-group by taking an l-sylow subgroup where l divides G_ K ur. 
Since $K^{ur}\subsetneq \overline{K}$, we have $G_{K^{ur}}\neq \{1\}$. Let $\ell$ be a prime dividing $|G_{K^{ur}}|$ (as a supernatural number). Let $S$ be a Sylow $\ell$-subgroup of $G_K$ and $L$ be the fixed field of $S$. It suffices to show that $\text{cd}(G_L)\geq 2$ since $G_L$ is a subgroup of $G_K$. We may thus replace $K$ with $L$ and assume that $G_K$ is an $\ell$-group to begin with. %algebraic extension of $K$ such that $G_L$ 
% and therefore $\text{cd}(G_{K^{ur}})\geq 1$. 
%Since $G_k\cong \widehat{\Z}$, we have $\text{cd}(G_k)=1$. 
By Proposition 22, pg. 28 \cite{serregal}, it follows that 
$$\text{cd}(G_K)=\text{cd}(G_{K^{ur}})+\text{cd}(G_k)\geq 1+1=2 $$
In fact, we have $\text{cd}(G_K)=2$ because $G_K$ is either a subgroup of $G_{\Q_p}$ or $G_{\F_p(\!(t)\!)}$, both of which have cohomological dimension $2$. 

\item Likewise, Lang's theorem is optimal in the sense that every non-maximal unramified extension of any local field is not $C_1$. Once again, this can be proved by noting that such a field has cohomological dimension $2$ (see Proposition 12, pg. 85 \cite{serregal}).
% (see Corollary, \S 2 \cite{Axdim}).

\item There are plenty of examples of algebraic extensions of local fields which are of cohomological dimension $1$ but not $C_1$ (see \cite{chipchakov}). 
%For a concrete example, first consider the compositum $k$ of all finite extension of $\F_5$ of degree prime to $6$. Let $L$ be the unramified extension of $\Q_5$ with residue field $k$. Arguing as in \cite{Axdim}, we see that the field $K=L(5^{1/n}:5\nmid n)$ is of cohomological dimension $1$ but is not $C_1$. 
It would be interesting to investigate to what extent the maximal unramified extension and the maximal totally ramified extensions are the only \textit{minimal} algebraic extensions of a local field which are $C_1$, i.e., which do not contain any proper subfields which are $C_1$.
\end{enumerate}
%https://www.maths.ed.ac.uk/~v1ranick/papers/bieri.pdf
%Probably also follows from Serre since p-co diemsnion depends only on the p-Sylow subgroup. 
\end{rem}

\subsubsection*{Acknowledgements} 
I wish to  thank Arno Fehm, Franz-Viktor Kuhlmann, François Loeser and Olivier Wittenberg for their comments on earlier versions of this paper. I am also grateful to Franziska Jahnke, Tom Scanlon, Micha\l \  Szachniewicz and Tingxiang Zou for several fruitful conversations, Jason Starr for pointing out his work on RC solving fields and Santai Qu for some clarifications.
%,  and  . 
\bibliographystyle{alpha}
\bibliography{references2}
\end{document}